\documentclass[12pt]{amsart}

\usepackage{cancel}
\usepackage{latexsym, amssymb, amsmath,esint}
\usepackage{soul}
\usepackage{amsfonts, graphicx}
\usepackage{graphicx,color}
\newcommand{\be}{\begin{equation}}
\newcommand{\ee}{\end{equation}}
\newcommand{\beq}{\begin{eqnarray}}
\newcommand{\eeq}{\end{eqnarray}}

\usepackage{wasysym,stmaryrd}
\newtheorem{thm}{Theorem}[section]

\newtheorem{lma}{Lemma}[section]

\newtheorem{cor}{Corollary}[section]
\newtheorem{defn}{Definition}[section]
\theoremstyle{remark}
\newtheorem{rem}{Remark}[section]
\numberwithin{equation}{section}

\newtheorem{claim}{Claim}[section]

\def\be{\begin{equation}}
\def\ee{\end{equation}}
\def\bee{\begin{equation*}}
\def\eee{\end{equation*}}
\def\ol{\overline}
\def\lf{\left}
\def\ri{\right}

\def\K{K\"ahler }

\def\Ric{\text{\rm Ric}}

\def\wh{\widehat}
\def\wt{\widetilde}

\def\p{\partial}

\def\ol{\overline}
\def\heat{\lf(\frac{\p}{\p t}-\Delta_{\eta}\ri)}
\def\tr{\operatorname{tr}}

\def\e{\varepsilon}

\def\a{{\alpha}}
\def\b{{\beta}}

\def\R{\mathbb{R}}

\begin{document}

\title[]
{Rigidity of area non-increasing maps}

 \author{Man-Chun Lee$^1$}
\address[Man-Chun Lee]{Department of Mathematics, The Chinese University of Hong Kong, Shatin, Hong Kong, China}
\email{mclee@math.cuhk.edu.hk}

\author{Luen-Fai Tam}
\address[Luen-Fai Tam]{The Institute of Mathematical Sciences and Department of Mathematics, The Chinese University of Hong Kong, Shatin, Hong Kong, China.}
 \email{lftam@math.cuhk.edu.hk}

\author{Jingbo Wan}
\address[Jingbo Wan]{Department of Mathematics, Columbia University, New York, NY, 10027}
 \email{jingbowan@math.columbia.edu}

 \thanks{$^1$Research partially supported by Hong Kong RGC grant (Early Career Scheme) of Hong Kong No. 24304222, 14300623, a direct grant of CUHK and NSFC grant No. 12222122.}

\renewcommand{\subjclassname}{
  \textup{2020} Mathematics Subject Classification}

\date{\today}

\begin{abstract}
In this work, we consider the area non-increasing map between manifolds with positive curvature. By exploring the strong maximum principle along the graphical mean curvature flow, we show that an area non-increasing map between certain positively curved manifolds is either homotopy trivial, Riemannian submersion, local isometry or isometric immersion. This implies that an area non-increasing self map of $\mathbb{CP}^n$, $n\ge 2$ is either homotopically trivial or is an isometry. This confirms a  speculation of Tsai-Tsui-Wang. We also use Brendle's sphere Theorem and mean curvature flow coupled with Ricci flow to establish related results on manifolds with positive $1$-isotropic curvature.
\end{abstract}

\maketitle

\markboth{Man-Chun Lee, Luen-Fai Tam, Jingbo Wan}{Rigidity of area non-increasing maps}

\section{Introduction}\label{s-introduction}

People have been interested to study properties of  maps between compact Riemannian manifolds in terms of the so-called $k$-dilation. The $1$-dilation maps (with constant 1), which are just distance non-increasing maps, have been studied intensively, see \cite{DeturckGluckStorm,Gromov1978,Hsu1972,Lawson1986} for example. The $2$-dilation maps (with constant 1)  are maps which are (two dimensional)  area non-increasing. In case the constant is less than 1, then it will be called (strictly) area decreasing map in this work. Interesting and important results have also been obtained. In particular using graphical mean curvature flow, Tsui-Wang \cite{TsuiWang2004} confirmed a conjecture of Gromov \cite{Gromov1996} that any rea {\it decreasing} map  from the standard sphere $\mathbb{S}^m$ into $\mathbb{S}^n$   must be homotopically trivial.
 Later Lee-Lee \cite{LeeLee2011}  proved that  any  area decreasing map between two compact manifolds satisfying certain curvature conditions  is also homotopically trivial,   see also \cite{AssimosSavasSmoczyk2023,SavasSmoczyk2014,SavasSmoczyk2018,LeeWan2023}  for more related results.  Recently, more general results are obtained by Tsai-Tsui-Wang \cite{TsaiTsuiWang2023} in this direction. It is well-known that in contrast,    by the work of Guth \cite{Guth2013}  one cannot expect an analogous result for $k\geq 3$. Motivated by the work of Tsai-Tsui-Wang \cite{TsaiTsuiWang2023}, in the first part of this work we want to study area non-increasing map under similar curvature conditions in their work. For area non-decreasing map, there is a well-known rigidity result by Llarull \cite{Llarull1998}  which says that any area non-increasing map from a {\it spin} $n$-manifold $(N^n,g)$ with scalar curvature $\mathcal{R}(g)\geq n(n-1)$  to the standard $n$-sphere with nonzero degree must be an isometry. Since the results of Tsai-Tsui-Wang \cite{TsaiTsuiWang2023} do not assume that the manifolds are spin,  one may wonder if there are similar results for area non-increasing map in the setting of \cite{TsaiTsuiWang2023}. Motivated by this, in the first part of this work, we prove the following:

\begin{thm}\label{t-intro-1}
 Let $(M^m, g), (N^n,h)$ be two compact manifolds with $m, n\ge 3$. Suppose $f_0$ is a smooth map from $M$ to $N$ which is area non-increasing. Suppose one of the following curvature conditions is satisfied:
 \begin{enumerate}
   \item [(i)]
\bee
 \Ric^g_{\min}-\Ric_{\max}^h+(m-\ell)\cdot \kappa_M+(n-\ell)\cdot \kappa_N\ge0, \;\;{\text and\ }\\[3mm]
 \kappa_M+\kappa_N> 0;
\eee
or
\item[(ii)]

\bee
\kappa_M> 0\ \text{and}\ (\ell-1)\tau_N\le  \lf(2(m-\ell)+\ell-1\ri)\kappa_M.
\eee
 with $\ell=\min\{m, n\}$.
 \end{enumerate}
 Here $\kappa_M, \kappa_N$ are the lower bounds of the sectional curvature of $M, N$ respectively, $\tau_N$ is the upper bound of the sectional curvature of $N$, $\Ric^g_{\min}$ is the minimum of the eigenvalues of $\Ric^g$ in $M$, and $\Ric^h_{\max}$ is the maximum of the eigenvalues of $\Ric^h$ in $N$.

Then either $f_0$ is homotopically trivial or $f_0$ is a Riemannian submersion (if $m>n$), local isometry (if $m=n$) or isometric immersion (if $m<n$).
\end{thm}
See Theorem \ref{t-TTW-2} for more details.  Among  other things, as a corollary, one can conclude that any area non-increasing self map from $\mathbb{CP}^n$, $n\ge 2$ with standard Fubini-Study metric is either an isometry or is homotopically trivial. This confirms a conjecture (`speculating' in their words) of Tsai-Tsui-Wang \cite{TsaiTsuiWang2023}. In particular, any area non-increasing self map from $\mathbb{CP}^n$ with non-zero degree must be an isometry.  On the other hand any area non-increasing map from $\mathbb{S}^m$ to $\mathbb{S}^n$ with $n>m\ge 3$ and with standard metrics must be either homotopically trivial or a Riemannian immersion or both. We should point out that the Hopf fibration $\mathbb{S}^{2n+1}\to \mathbb{CP}^n$ is area nonincreasing. But the map does not satisfy either condition in the above theorem. In this sense, the theorem is not sharp. We should also remark that the conditions (i), (ii) above are not totally unrelated. In fact, one can check that $(\ell-1)\tau_N\le  \lf(2(m-\ell)+\ell-1\ri)\kappa_M$ will imply $\Ric^g_{\min}-\Ric_{\max}^h+(m-\ell)\cdot \kappa_M+(n-\ell)\cdot \kappa_h\ge0$. Moreover, condition (i) will imply that $\Ric^g\ge 0$.

The second main theme of this work is to generalize the work of Tsui-Wang \cite{TsuiWang2004} on spheres in a different direction.   Let us first recall the definition of positive $1$-isotropic curvature.

\begin{defn}\label{defn:PIC1}
We say that a curvature type tensor $R$ is in the cone of positive $1$-isotropic curvature, i.e. $R\in  \mathrm{C}_{PIC1}$, if for all orthonormal frame $\{e_i\}_{i=1}^4$ and $\mu\in [0,1]$, we have
$$R_{1331}+\mu^2 R_{1441}+R_{2332}+\mu^2 R_{2442}-2\mu R_{1234}> 0.$$
For a Riemannian manifold $(M,g)$, we define
\begin{equation}
\chi_{IC1}(g)=\sup\left\{ s\in \mathbb{R}: \mathrm{Rm}(g)-s\cdot \frac12 g\owedge g\in \mathrm{C}_{PIC1}\right\}
\end{equation}
where $\owedge$ denotes the Kulkarni-Nomizu product, see \eqref{e-KNproduct}. We say that $(M,g)$ has non-negative $1$-isotropic curvature if $\chi_{IC1}(g)\geq 0$.
\end{defn}

 One can equivalently interpret the non-negative $1$-isotropic curvature using the language of Lie algebra by the work of Wilking \cite{Wilking2013}.
In dimension three, the latter definition is still well-defined and is equivalent to $\Ric\geq 0$. Therefore, $\chi_{IC1}(g)\geq 0$ will be understood to be $\Ric\geq 0$ when $n=3$.

In \cite{Brendle2008}, Brendle showed that compact manifolds with positive $1$-isotropic curvature are diffeomorphic to quotient of sphere using the Ricci flow, which is a generalization of Brendle-Schoen's differentiable sphere theorem in \cite{BrendleSchoen2009}.
 Motivated by Brendle's Sphere Theorem,  we are interested in the homotopy problem on manifolds with positive $1$-isotropic curvature.
By considering the graphical mean curvature flow coupled with the Ricci flows,  we prove the following theorem:
\begin{thm}\label{t-intro-2}
 Let $(M^m, g)$ be a compact, locally irreducible and locally non-symmetric manifold with $\chi_{IC1}(g)\geq 0$ and $(N^n,h)$ be a compact manifold such that $m,n\geq 3$. Suppose $f_0$ is a smooth map from $M$ to $N$ which is area non-increasing. Suppose
 \begin{equation}
 \mathcal{R}_{min}(g)\geq \frac{m}{n}\mathcal{R}_{max}(h)
 \end{equation}
 and one the following curvature conditions is satisfied
 \begin{enumerate}
   \item [(i)] $(N,h)$ is Einstein with $\kappa_N\geq 0$;
or
\item[(ii)]
$
\tau_N\leq  0$.
 \end{enumerate}
 Here $\mathcal{R}_{min}(g)$ denotes the minimum of scalar curvature of $g$ and $\mathcal{R}_{max}(h)$ denotes the maximum of scalar curvature of $h$. Then either $f_0$ is homotopically trivial or $f_0$   local isometry (if $m=n$) or isometric immersion (if $m<n$).
\end{thm}
See Theorem~\ref{t-pic-rigidity} for more details. If $g$ is assumed to have $\chi_{IC1}(g)>0$, i.e. positive $1$-isotropic curvature, then $g$ is locally irreducible. Then $g$ must be either locally non-symmetric or a quotient of sphere by Brendle's Sphere Theorem \cite{Brendle2008}. In either case, we might apply Theorem~\ref{t-intro-1} or \ref{t-intro-2} to discuss the area non-increasing maps $f$ from $M$ to $N$.

\medskip
 The paper is organized as follows. In Section 2, we discuss the preliminary on geometry of graphical mean curvature flow, its short-time and long-time existence. In Section 3, we will establish monotonicity along the graphical mean curvature under various curvature conditions and will use it to prove the rigidity of area non-increasing maps. Section 3 is further divided into three parts: evolution equations, monotonicity in static backgrounds, and monotonicity in evolving backgrounds. The main strong maximum principle type argument is presented in the proof of Theorem \ref{t-TTW-2} within a static background, emphsizing curvature condition (i) from Theorem \ref{t-intro-1}. This approach extends to various cases in both static and evolving backgrounds. To prove Theorem \ref{t-intro-2}, we incorporate evolving backgrounds via Ricci Flow and apply the strong maximum principle alongside Brendle's Sphere Theorem (see Section 3.3 for details).

\vskip0.2cm

{\it Acknowledgement}: The authors would like to thank Prof. Chung-Jun Tsai, Prof. Mao-Pei Tsui and Prof. Mu-Tao Wang for valuable discussions and answering some of the questions. The first named author would like to thank Prof. Jason Lotay for some insightful discussion on coupled flows. The third named author would like to thank Prof. Mu-Tao Wang for his continuing support, and also for introducing this problem.

\section{Preliminaries: graphical mean curvature flow, short time and long time existence}\label{s-existence}\label{sec:graph-geom}

Let $(M^m,g(t))$, $(N^n,h(t))$ be two compact manifolds with smooth families of metrics $g(t), h(t)$ which may be independent of time. In this section, we discuss the mean curvature flow from $M$ to   $X^{m+n}=M^m\times N^n$  with $G(t)=g(t)\oplus h(t)$ is the product metric. We will concentrate on graphical solutions. The short time existence is well-known \cite[\S 6.4]{AndrewsChowGuentherLangford} and the `long time' existence is now standard because of the work of Wang \cite{Wang2002}.  However,  we will sketch the proofs on the existence the solution  for the convenience of the readers. First let us recall the mean curvature flow equation.

\subsection{Mean curvature flow equation and short-time existence}\label{ss-shortime}

Let us first recall the setting of the mean curvature flow.
Let $(M^m, \eta(t))$, $(X^Q,G(t))$ be a compact manifolds where   $G(t)$ is a smooth family of metrics on $X$, $t\in [0,T)$, $T>0$. Let $\nabla$, $\wt\nabla$ be Riemannian connections on $(M,\eta(t))$,  $(X,G(t))$. Let $F:\wt M=: M\times[0,T)\to X$ be a smooth map.
Consider in local coordinates $\{x^i\}$ in $M$, $\{y^\a\}$ in $X$, a section of $\otimes^k(T^*(\wt M))\otimes F^{-1}(T(X))$ is of the following form:
$$
s=s^\a_{i_1\dots i_k}dx^{i_1}\otimes\dots\otimes dx^{i_k}\otimes \p_{y^\a}.
$$
Then at $(x,t)$:
\be
\begin{split}
s_{|p}=&
D_{\p_{ x^p}}s\\
=&s^\a_{i_1\dots i_k;p} dx^{i_1}\otimes\dots\otimes dx^{i_k}\otimes \p_{y^\a}+s^\a_{i_1\dots i_k}dx^{i_1}\otimes\dots\otimes dx^{i_k}\otimes\wt\nabla_{ F_*(\p _{x^p})} \p_{y^\a}\\
=&s^\a_{i_1\dots i_k;p} dx^{i_1}\otimes\dots\otimes dx^{i_k}\otimes \p_{y^\a}+F^\b_ps^\a_{i_1\dots i_k}dx^{i_1}\otimes\dots\otimes dx^{i_k}\otimes\wt\nabla_{\p_{y^\b}} \p_{y^\a}
\end{split}
\ee
Here $;$ is the covariant derivative with respect to  $\eta(t)$. If there is no confusion, we will also use $;$ to denote the covariant derivative with respect to $G(t)$.  Also,
\be
\begin{split}
s_{|t}=&
D_{\p_{t}}s\\
=&\p_ts^\a_{i_1\dots i_k} dx^{i_1}\otimes\dots\otimes dx^{i_k}\otimes \p_{y^\a}+s^\a_{i_1\dots i_k}dx^{i_1}\otimes\dots\otimes dx^{i_k}\otimes\wt\nabla_{F_*(\p t)} \p_{y^\a}\\
=&\p_t s^\a_{i_1\dots i_k} dx^{i_1}\otimes\dots\otimes dx^{i_k}\otimes \p_{y^\a}+F^\b_ts^\a_{i_1\dots i_k}dx^{i_1}\otimes\dots\otimes dx^{i_k}\otimes\wt\nabla_{\p_{y^\b}} \p_{y^\a}
\end{split}
\ee
Suppose $F_0:M\to X$ be an immersion. The mean curvature flow equation is given by:
\be\label{e-mcf-1}
\left\{
  \begin{array}{ll}
    \p_t F=H, & \hbox{on $M\times[0,T)$.}\\
F|_{t=0}=F_0.
  \end{array}
\right.
\ee
where  $F_t:M\to X$ with $F_t(x)=F(x,t)$ is an immersion, and    $H=H(t)$ is the mean curvature vector of $F_t(M)$ with respect to $G(t)$. Let $\eta(t)=F_t^*(G(t))$.
We define $\Delta_{\eta(t)}s=\eta^{ij}s_{|ij}$ then the mean curvature flow is of the form:
\be\label{e-mcf-2}
 \p_tF=\Delta_{\eta(t)}F. 
\ee
We have the following short time existence result.

\begin{lma}\label{l-shortime} In the above setting, let $F_0:M\to (X, G(0))$ be a smooth immersion. Then there is $T_0>0$ such that \eqref{e-mcf-1} has a solution on $M\times[0,T_0)$. Namely, $F_t$ is an immersion, and satisfies \eqref{e-mcf-2} with $\eta(t)=F_t^*(G(t))$. Moreover, the solution is unique.
\end{lma}
\begin{proof} The proof is exactly as the case when $G(t)$ is a fixed metric, see \cite[\S6.4]{AndrewsChowGuentherLangford}.  We sketch the proof. Let $h$ be a fixed metric on $M$ with connection $\ol \Gamma$. Consider the parabolic systems in local coordinates

\be\label{e-Deturk}
\p_t F^\a=\eta^{ij}\lf(F^\a_{ij}-\ol\Gamma_{ij}^kF^\a_k+\Gamma^\a_{\b\gamma}F^\b_i F^\gamma_j\ri).
\ee
This is a global solution: it does not depend on local coordinates, as
\bee
 \p_t F^\a=\eta^{ij}\lf(F^\a_{ij}-\Gamma_{ij}^kF^\a_k+\Gamma^\a_{\b\gamma}F^\b_iF^\gamma_j\ri)
+\eta^{ij}(\Gamma_{ij}^k-\ol\Gamma_{ij}^k)F^\a_k
\eee

Since this is strictly parabolic and $M$ is compact,  it admits a short time solution with initial map $F_0$ by standard parabolic theory. If $t$ is small, then $F_t(p)=F(p,t)$ is also immersion.  Let us assume this is the case. Consider the  vector field $W(p,t)$ such that
$dF_t(W)=\eta^{ij}( \Gamma_{ij}^k-\ol\Gamma_{ij}^k)F^\a_k\p_{y^\a}$
which is tangent to $F_t(M)$. Since $F_t$ is an immersion, $W$ exists. Consider the solution of ODE:
\be\label{e-ODE}
\left\{
  \begin{array}{ll}
    \frac{d}{dt}\phi(p,t)=-W(\phi(p,t),t)\\
\phi(p,0)=p.
  \end{array}
\right.
\ee
Then $p\mapsto \phi(p,t)$ is a diffeomorphism on $M$. Let
\bee
\wt F(p,t)=F(\phi(p,t),t).
\eee
Then
\bee
\begin{split}
\p_t \wt F(p,t)=&\p_t F(\phi(p,t),t)+dF_t\circ d\phi (\p_t)\\
=& \p_t F(\phi(p,t),t)-dF_t(W) \\
=&\Delta_{\eta(t)}F|_{ \phi(p,t)}\\
=&\Delta_{\phi_t^*(\eta(t))}\wt F|_{\phi_t^{-1}(\phi(p,t))}\\
=&\Delta_{\wt \eta(t)}\wt F
\end{split}
\eee
 where $\eta(t)=F_t^*G(t)$ and $\wt \eta(t)=\wt F_t^*G(t)$. The uniqueness follows from the fact that both solutions to \eqref{e-Deturk} and \eqref{e-ODE} are unique provided the initial data are the same,  see \cite[\S6.4]{AndrewsChowGuentherLangford}.
\end{proof}

\subsection{Geometry of graph and area non-increasing maps}\label{subsec:graph-geom}

We now specialize to the case when $X=M^m\times N^n$, $G(t)=g(t)\oplus h(t)$ is an evolving product metric and the initial data is a graph, i.e. $F_0=\mathrm{Id}\times f_0 {: M\rightarrow X }$. Let $F:M\times [0,T]\to X$ be the solution obtained from  Lemma~\ref{l-shortime}, it is easy to see that the solution remains graphic for a short time. We assume working on $[0,T]$ where $F$ remains graphical and will therefore call  it to be a solution to the graphical mean curvature flow coupled with $G(t)$. In what follows, we will abbreviate it as graphical mean curvature flow when the context is clear.  By graphical condition, there exists smooth family of diffeomorphism $\phi_t\in \mathrm{Diff}(M)$ and maps $f_t:M\to N$ such that $F_t=\left(\mathrm{Id}\times f_t\right)\circ \phi_t$ for each $t\in [0,T]$. At each $(p,t)\in M\times [0,T]$, we let $\lambda_1^2\geq ...\geq \lambda_m^2$ be the eigenvalues of $f_t^*h$ with respect to $g$ at $x=\phi_t(p)$ so that the corresponding $\lambda_i\geq 0$ are the singular values of $df_t$ at $x$. 
The map $f_t:(M,
g(t))\to (N,h(t))$
is said to be:
\be\label{e-defn-area-decreasing}
\left\{
  \begin{array}{ll}
    \text{distance non-increasing, if $\lambda_i\le 1$ for all $i$};\\
\text{distance decreasing, if $\lambda_i< 1$ for all $i$};\\
\text{area non-increasing, if $\lambda_i\lambda_j\le 1$ for all $i\neq j$};\\
\text{area decreasing, if $\lambda_i\lambda_j< 1$ for all $i\neq j$}.\\
  \end{array}
\right.
\ee
The meaning of the terminology is obvious. In order to emphasis, sometimes we will call area decreasing map as strictly area decreasing map. Following \cite{TsuiWang2004}, to detect whether $f_t$ is area non-increasing, we introduce the following tensors to detect whether $f_t$ is distance non-increasing etc. Let $\pi^M, \pi^N$ the projections of $M\times N$ onto $M, N$ respectively. Let
\begin{equation}
s(t)=\pi_M^*g(t)-\pi_N^*h(t),
\end{equation}
and  let $S(t)=F_t^*s(t)$ be a $2$-tensor on $M$ which in local coordinate is given by
\begin{equation}
S_{ij}=F^\a_i F^\b_j s_{\a\b}.
\end{equation}
If there is no possible confusion, we will also $\pi_M^*g(t)$ by $g(t)$ for example.

Let $\Theta= S\owedge \eta$.  In local coordinate,
\begin{equation}
\Theta_{ijkl}=S_{il} \eta_{jk}+S_{jk}\eta_{il}-S_{ik}\eta_{jl}-S_{jl}\eta_{ik}.
\end{equation}
Here  $\owedge$ be the Kulkarni-Nomizu product: namely, for symmetric $(0,2)$ tensors $S, T$ on a manifold,
\be\label{e-KNproduct}
\begin{split}
(S\owedge T)(X,Y,Z,W):=&S(X,W)T(Y,Z)+S(Y,Z)T(X,W)\\
&-S(X,Z)T(Y,W)-S(Y,W)T(X,Z)
\end{split}
\ee
for tangent vectors $X,Y,Z, W$. Equivalently one can write
$$
\Theta=F_t^*g\owedge F_t^*g-F_t^*h\owedge F_t^*h.
 $$
Furthermore, $\Theta$ is a curvature type tensor and   can be considered as a symmetric bilinear form on $\Lambda^2 T^*M$ so that
$$\Theta(X\wedge Y,Z\wedge W)=\Theta(X,Y,W,Z)=(S\owedge \eta)(X,Y,W,Z).$$

Hence if $e_1, e_2$ is an orthonormal pair, then
$$\Theta(e_1\wedge e_2,e_1\wedge e_2)=\Theta(e_1,e_2,e_2,e_1).$$

In our convention, the sectional curvature of the two plane spanned by orthonormal pair $e_1, e_2$ is given by
$$
R(e_1,e_2,e_2,e_1).
$$

If $g$ is a Riemannian metric, then
$$
(g\owedge g)(X,Y,Y,X)=2\lf[g(X,X)g(Y,Y)- \lf(g (X,Y)\ri)^2 \ri]
$$
which is just twice of the area of the parallelogram spanned by $X,Y$. Hence
it is easy to see that $S\ge 0$ with respect the metric $\eta(t)$ if and only if $f_t$ is distance non-increasing, and $\Theta\ge 0$ as a bilinear form on   $\Lambda^2 T^*M$ with respect the metric $\frac12\eta\owedge\eta$ { if and only if $f_t$ is area non-increasing}. More precisely, we have:
\begin{lma}
Given a graphical embedding $F:M\to M\times N$ such that $F=(Id\times f)\circ \phi$ for some map $f:M\to N$ and diffeomorphism $\phi\in \mathrm{Diff}(M)$, then
\begin{enumerate}
\item[(i)] $f$ is distance non-increasing (decreasing resp.) if and only if $S\geq 0$ ($>0$ resp.) on $M$;
\item[(ii)] $f$ is area non-increasing (decreasing resp.)  if and only if $\Theta\geq 0$ ($>0$ resp.) on $M$.
\end{enumerate}
\end{lma}

We will also use the notation that $\Theta\ge a$ to mean that $\Theta-\frac a2\eta\owedge\eta\ge0$ as a symmetric bilinear form on $\Lambda^2 T^*M$.

To ease our computation, it is easier to choose a frame which is compatible with the graphical structure. To do this, for $p\in M$ we let $\{u_i\}_{i=1}^n$ be a $g$-orthonormal frame at $x=\phi_t(p)\in M$ and $\{v_i\}_{i=1}^m$ be a $h$-orthonormal frame at $f_t(x)\in N$ so that $df_t(u_i)=\lambda_i v_i$ for all $1\leq i\leq \min\{m,n\}=:\ell$.  We extend $\lambda_i$ to be $0$ for $i>\min\{n,m\}$ for notation convenience.  For metrics $g(t)$ and $h(t)$, we will then use $K^g_{ij}, K^h_{ij}$ to denote the sectional curvature of the plan $\Sigma=\mathrm{span}\{u_i,u_j\}$ and $\Sigma'=\mathrm{span}\{ v_i,v_j\}$.    Their curvature and evolution of metrics  are defined analogously.

With $\{u_i\}_{i=1}^m$ and $\{v_i\}_{i=1}^n$, we define
\begin{equation}\label{frame-onGraph}
\left\{
\begin{array}{ll}
e_i=\displaystyle \frac{u_i+\lambda_i v_i}{\sqrt{1+\lambda_i^2}},\;\; \text{for}\;\; 1\leq i\leq m\\[5mm]
\nu_a=\displaystyle \frac{-\lambda_a u_a+v_a}{\sqrt{1+\lambda_a^2}}, \;\; \text{for}\;\; 1\leq a\leq n
\end{array}
\right.
\end{equation}
so that $\{e_i\}_{i=1}^m$ forms an orthonormal basis on $T_{(x,f_t(x))}F_t(M)$ while $\{\nu_a\}_{a=1}^n$ forms an orthonormal basis on $N_{(x,f_t(x))}F_t(M)$ with respect to the product metric $G$.  Thus, $d\pi_M (e_i)=(1+\lambda_i^2)^{-1/2} u_i$ and hence the map $f_t$ will stay graphical if $\lambda_i$ is bounded uniformly. To unify the notation, we will denote this orthonormal frame as $\{\tilde E_\a\}_{\a=1}^{m+n}$ by defining $\tilde E_\a=e_\a$ for $1\leq \a\leq m$ and  $\tilde E_\a=\nu_a$ for $\a=m+a$ where $1\leq a\leq n$. With respect to this singular decomposition, we then have $\eta_{ij}=\delta_{ij}$ and
\begin{equation}\label{eqn:graph-frame-1}
\left\{
\begin{array}{ll}
S_{ij}=\displaystyle \frac{1-\lambda_i^2}{1+\lambda_i^2}\delta_{ij};\\[3mm]
\Theta_{ijji}=S_{ii}+S_{jj}
\text{for}\;\; i\neq j;\\
C_{ii}=:\displaystyle\frac{2\lambda_i}{1+\lambda_i^2}
\end{array}
\right.
\end{equation}
at the point $(x_0,t_0)\in M\times [0,T]$ under consideration.   Note that $S_{ii}^2+C_{ii}^2=1$ for $1\leq i\leq m$. 

Since $e_i\wedge e_j$ for $i<j$ are orthonormal frames of { $\Lambda^2T^*M$ with respect to} the metric $\frac12\eta\owedge \eta$, one can see that $$\{\Theta_{ijji}|\ \  1\leq i<j\leq m\}$$  are indeed the eigenvalues of $\Theta$ with respect to  $\frac12\eta\owedge \eta$.   Furthermore for $1\leq i,j\leq m$ and $1\leq a,b\leq n$,
\begin{equation}\label{eqn:graph-frame-2}
\left\{
\begin{array}{ll}
s(\tilde E_i,\tilde E_j)=s(e_i,e_j)=S_{ij};\\[3mm]
s(\tilde E_i,\tilde E_{m+a})=\displaystyle \frac{-2\lambda_i}{1+\lambda_i^2}\delta_{ia};\\[3mm]
s(\tilde E_{m+a},\tilde E_{m+b})=\displaystyle\frac{\lambda_a^2-1}{1+\lambda_a^2}\delta_{ab}.
\end{array}
\right.
\end{equation}
We abbreviate it as graphical frame. We will use it when applying tensor maximum principle.

\subsection{Long-time existence}\label{ss-longtime}
Recall that a graphical mean curvature flow is given by a graph of $f_t: (M,g(t))\to (N,h(t))$ where one can use the graphic frame \eqref{frame-onGraph}. Let $\lambda_i$ be the eigenvalues of $f_t^*(h(t))$ as in the definition of the graph frame and let $\Theta$ be as in \eqref{eqn:graph-frame-1}. Following the argument of Wang \cite{Wang2002} which is based on Huisken's monotonicity \cite{Huisken1990}, we prove the following long-time existence criteria with possibly evolving background.

\begin{thm}\label{thm:long-time-MCF-RF}
Let $(M^m,g(t))$ and $(N^n,h(t))$ be two smooth compact evolving families on $[0,T)$ with $T\le +\infty$. Denote $G(t)=g(t)\oplus h(t)$ be the product metric on $X=M\times N$. Suppose $F:M\times [0,t_0)\to (X,G(t))$ is a smooth solution of the  mean curvature flow coupled with $G(t)$ with $t_0<T$.  If for each $t$, $F_t$ is a graph given by strictly area decreasing map $f_t:M\to N$ so that $\Theta(\cdot,t)\geq \delta$ for some $\delta>0$ on $[0,t_0)$, then the flow can be extended beyond $t_0$ and remains graphical.
\end{thm}

\begin{proof}
 Since  $t_0<T$,  both $g(t)$ and $h(t)$ has bounded geometry of infinity order on $[0,t_0]$.  Moreover since $\Theta\geq \delta>0$ on $M\times [0,t_0)$,  we have $0\leq \lambda_i\leq \frac2\delta$ by \eqref{eqn:graph-frame-1}  on $M\times[0,t_0)$, see \cite[Lemma 3.3]{TsaiTsuiWang2023} for example.

   In what follows, we will use $C_i$ to denote any positive constant  depending only on $G|_{[0,T_0]},\delta, n,m, X$. By the proof of \cite[Theorem 3.2]{TsaiTsuiWang2023}  (see Corollary \ref{c-det} below), and the fact that  $\lambda_i\leq \frac2\delta$, we have
\bee
\heat \log \lf(\frac{\det\Theta}{\frac12\eta\owedge\eta}\ri)\ge a |A|^2-C_2.
\eee
for some $a>0$ depending only on $m, n$ and $|A|$ is the norm of the second fundamental form. Here $\eta(t)=F^*_t(G(t))$. Since $0\ge \log \frac{\det\Theta}{\frac12\eta\owedge\eta}$ and $\Theta\ge \delta$, we can find $C>0$ such so that   $\phi=C-\log \lf(\frac{\det\Theta}{\frac12\eta\owedge\eta}\ri) $ satisfies  $1\leq \phi\leq C$ and
\be\label{e-phi-longtime}
\heat\phi\le -a|A|^2+C_2.
\ee

We want to follow the argument of Wang \cite{Wang2002}. We isometrically embed $\left(X, G(t_0)\right)$ into $\mathbb{R}^Q$. Denote its image also by $X$. In a tabular neighborhood $U$ of $X$, extend $G(t)$ smoothly to $\wt G(t)$ in $U$ so that $\wt G(t_0)$ is the Euclidean metric on $U$. This can be done  since $G(t)$ is smooth on $X$. Now we re-write the mean curvature flow equation as
\begin{equation}\label{e-mcf-3}
\partial_t F= H=\wt H+ E_1
\end{equation}
where   $E_1=-\tr_{F_t(M)}A_{ (U,\wt G(t));(X,G(t))}$ and $\wt H$ is the mean curvature of $F_t(M)$ in $(U,\wt G(t))$ and $A_{ (U,\wt G(t));(X,G(t))}$ is the second fundamental form of $(X,G(t))$ in $(U,\wt G(t))$. We consider the ($m$-dimensional) backward heat kernel in $\mathbb{R}^Q$ centered at $(y_0,t_0)$: for $t<t_0$ and $y\in \mathbb{R}^Q$,
\begin{equation}
\rho_{y_0,t_0}(y,t)=\frac1{ (4\pi (t_0-t))^{m/2}}\cdot \exp\left(-\frac{d^2_{euc}(y,y_0)}{4(t_0-t)} \right).
\end{equation}
where $d_{euc}(x,y)=|x-y|$ denotes the standard Euclidean distance between $x,y\in \mathbb{R}^Q$.  We will also use $\bullet$ to denote the Euclidean inner product and denote $\tau(t)=t_0-t$. By translation, we assume $y_0$ to be the origin.  Consider the function
$
\rho(x,t)=\rho_{y_0,t_0} \left( F(x,t),t\right)$
on $M\times [0,t_0)$.  By differentiating $\rho$ with respect to $x,t$, we have
\begin{equation}
\left\{
\begin{array}{ll}
\partial_t \rho&\displaystyle=\rho \left[ \frac{m}{2\tau}-\frac{|F|^2}{4\tau^2}-\frac{\partial_t F \bullet F}{2\tau}\right];\\[5mm]
\rho_i&\displaystyle=-\rho \cdot \frac{F_i\bullet F}{2\tau};\\[5mm]
\rho_{;ij}&\displaystyle=\rho \cdot \frac{F_i\bullet F}{2\tau}  \frac{F_j\bullet F}{2\tau} -\rho \cdot \frac{F_{;ij}\bullet F +F_i\bullet F_j}{2\tau}
\end{array}
\right.
\end{equation}
where $;$ is the covariant derivatives with respect to $\eta=F^*G$ which is the induced metric of $G(t)$ on the submanifold $F_t(M)$.  Hence,
\begin{equation}
\begin{split}
 \left( \frac{\partial}{\partial t} +\Delta_\eta \right) \rho
&=\rho \left[ \frac{m}{2\tau}-\frac{\eta^{ij}F_i\bullet F_j}{2\tau}+ \frac{\eta^{ij}(F_i\bullet F)( F_j\bullet F) }{ 4\tau^2} -\frac{|F|^2}{4\tau^2}\right]\\
&\quad -\rho \cdot \frac{\lf((\partial_t+\Delta_{\eta,euc})F\ri) \bullet F }{2\tau}.
\end{split}
\end{equation}

Since $(X,G(t))$ is isometrically embedded into $(\mathbb{R}^Q,\wt G(t))$ and $\wt G(t_0)$ is the Euclidean metric, for the orthonormal frame $\{e_i\}_{i=1}^m$ with respect to $\eta(t)$, $F_i=dF(e_i)$ so that
\begin{equation}
\begin{split}
F_i \bullet F_i &=\wt G({t_0})\left( dF(e_i), dF(e_j) \right)\\
&=\wt G(t)\left( dF(e_i), dF(e_j) \right)+O(\tau)\\
&=\eta(e_i,e_j)+ O(\tau)\\
&=\delta_{ij}+O(\tau)
\end{split}
\end{equation}
because $dF$ is bounded uniformly with respect to $\wt G$ and $\wt G$ varies smoothly in $t$. Therefore,
\begin{equation}
\begin{split}
\frac{m}{2\tau}-\frac{\eta^{ij}F_i \bullet F_j}{2\tau}\leq C_3.
\end{split}
\end{equation}

On the other hand, if we denote $\zeta= F_t^* G(t_0)$ and decompose $F$ (as a vector in $\R^Q$) using Euclidean metric $\wt G(t_0)$:
\begin{equation}
F=F^\perp + a^i F_i
\end{equation}
where $a^i= \zeta^{ij} F\bullet F_j$. Here $F^\perp$ is the normal part of $F$ on $F_t(M)$ with respect to the Euclidean metric.  Since $G(t)\to G(t_0)$ smoothly as $t\to t_0$,  $\zeta^{ij}=\delta^{ij}+O(|t_0-t|)$ and thus
\begin{equation}
\begin{split}
|F|^2=F\bullet F&= [ F^\perp + (\zeta^{ij} F\bullet F_i) \cdot F_j]\bullet [ F^\perp + (\zeta^{kl} F\bullet F_k) \cdot F_l]\\
&=|F^\perp|^2 +\sum_{i=1}^m (F_i \bullet F)^2 + O(\tau)\cdot |F|^2.
\end{split}
\end{equation}

We now simplify $(\partial_t+\Delta_{\eta,euc})F\bullet F$ further by first noting that $\Delta_{\eta,euc}F\perp F_t(M)$ with respect to the Euclidean metric $\wt G(t_0)$,  and
\bee
\begin{split}
\Delta_{\eta,euc}F^\a&=\Delta_{\eta,\wt G} F^\a+\eta^{ij} \Psi_{\b\gamma}^\a F^\b_i F^\gamma_j
\end{split}
\eee
where $\Psi= \Gamma^{\wt G}(t)-\Gamma^{euc}=\Gamma^{\wt G(t)}-\Gamma^{\wt G(t_0)}=O(\tau)$. Hence
\be
\begin{split}
\Delta_{\eta,euc}F=\wt H +E_2\\
\end{split}
\end{equation}
where  $E_2=O(\tau)$ because the energy density of $F_t$ as a map from $(M,\eta(t))$ to $(X,G(t))$ is $m$ and $G(t)$ is uniformly equivalent to $G(t_0)$ which is induced by the  Euclidean metric. Therefore using \eqref{e-mcf-3}, we deduce
\begin{equation}
\begin{split}
\lf((\partial_t+\Delta_{\eta,euc}) F\ri)\bullet F &=(H+\wt H+E_2) \bullet F\\
&=(2\wt H +E_1) \bullet F^\perp+O(\tau) \cdot |F|
\end{split}
\end{equation}
where we have used $\Delta_{\eta,euc}F\perp F_t(M)$, $E_2=O(\tau)$, $G(t)\to G(t_0)$ as $t\to t_0$ smoothly and $E_1\in N(X)$ in $U$ with respect to $G(t)$.

Hence, we have
\begin{equation}\label{eqn:rho}
\begin{split}
 \left( \frac{\partial}{\partial t} +\Delta_\eta \right) \rho&\leq \rho\left(C_4+\frac{C_5|F|^2}{\tau} -\frac{|F^\perp|^2}{4\tau^2}-\frac{ (\wt H +\frac12 E_1)\bullet F^\perp}{\tau }\right).
\end{split}
\end{equation}

We are  now in position to apply Wang's argument \cite{Wang2002}. We start by observing the Gaussian density is bounded from the  $C^1$ bound of $f_t$.
\begin{claim} \label{claim:gas-bdd}
For any $\a>0$, there exists $C_\a>0$ such that for all $t\to t_0$,
\begin{equation}
\int_{M} (t_0-t)^{-m/2}\exp\left(-\frac{|F(x,t)|^2}{\a(t_0-t)}  \right)\,d\mathrm{vol}_{\eta(t)}(x)\leq C_\a.
\end{equation}
\end{claim}

Suppose the claim is true,
first observe that
\begin{equation*}
\begin{split}
\frac12\eta^{ij}\partial_t \eta_{ij}&=-|H|_{G^2(t)} -\eta^{ij} J(dF(\partial_i), dF(\partial_j))\\
&\leq  -  \wt G(t)(H,\wt H)+C_6\\
&=-\wt G(t)(\wt H+E_1, \wt H)+C_6
\end{split}
\end{equation*}
since $J=\partial_t G$ is bounded and $E_1\perp H$ with respect to $\wt G(t)$. Combining with \eqref{eqn:rho}, \eqref{e-phi-longtime} and the fact that $\phi\ge 1$ is bounded, we have
\begin{equation}
\begin{split}
&\frac{d}{dt}\int_{M} \phi  \rho \,d\mathrm{vol}_{\eta(t)}\\
&=\int_{M} \left( \partial_t \phi \cdot  \rho+ \phi \cdot \partial_t \rho+\frac12\phi\rho\cdot \tr_\eta (\partial_t \eta) \right) \,d\mathrm{vol}_{\eta}\\
&=\int_{M}  \left(\frac{\partial}{\partial t}-\Delta_\eta \right) \phi \cdot \rho+\phi \cdot\left(\frac{\partial}{\partial t}+\Delta_\eta \right)  \rho  \,d\mathrm{vol}_{\eta}\\
&\quad +\frac12 \int_M\phi\rho\cdot   \tr_\eta \partial_t\eta \,  \,d\mathrm{vol}_\eta\\
&\leq   - a\int_M \rho |A|^2 \rho \,d\mathrm{vol}_\eta+C_7\int_M\rho(1+\frac{|F|^2}\tau)  d\mathrm{vol}_\eta\\
&+\int_M \phi \rho\bigg[-\frac{|F^\perp|^2}{4\tau^2}-\frac{ (\wt H +\frac12 E_1)\bullet F^\perp}{\tau } -\wt G(t)(\wt H+E_1, \wt H)\bigg]\, d\mathrm{vol}_\eta.
\end{split}
\end{equation}
On the other hand,
\bee
\begin{split}
\wt G(t)(\wt H+E_1, \wt H)\ge &\wt G(t_0)(\wt H+E_1,\wt H)-O(\tau)|\wt H+E_1|_{euc} |\wt H|_{euc}\\
\ge & (\wt H+E_1)\bullet\wt H- O(\tau)(|\wt H+\frac12 E_1|^2_{euc} +|E_1|^2_{euc})\\
\ge&(1-C_9\tau)|\wt H+\frac12 E_1|^2_{euc} -C_8
\end{split}
\eee
Hence
\bee
\begin{split}
&-\frac{|F^\perp|^2}{4\tau^2}-\frac{ (\wt H +\frac12 E_1)\bullet F^\perp}{\tau } -\wt G(t)(\wt H+E_1, \wt H)\\
\le& -\frac{|F^\perp|^2}{4\tau^2}-\frac{ (\wt H +\frac12 E_1)\bullet F^\perp}{\tau }-(1-C_9\tau)|\wt H+\frac12 E_1|^2_{euc} +C_8\\
\le &C_{10}\frac{|F|^2}\tau.
\end{split}
\eee
Therefore,
\bee
\frac{d}{dt}\int_{M} \phi  \rho \,d\mathrm{vol}_{\eta(t)}\le C_{12}\int_{M}   \rho (1+\frac{|F|^2}\tau) \,d\mathrm{vol}_{\eta(t)}\le C_{13}
\eee
by the  Claim~\ref{claim:gas-bdd}.
This shows that
$
\lim_{t\to t_0} \int_M \phi \rho \,d\mathrm{vol}_\eta$ exists by monotone convergence Theorem. In the above argument, if we replace $\phi$ by $1$, we also have
\begin{equation}\label{monot-Gauss}
\frac{d}{dt}\int_{M} \rho \,d\mathrm{vol}_{\eta}\leq C_{14}.
\end{equation}
So  $\lim_{t\to t_0}\int_{M} \rho \,d\mathrm{vol}_{\eta}$ exists.
 Now we can follow the argument in \cite{Wang2002} to deduce
\begin{equation}
\lim_{t\to t_0}\int_{F_t(M)} \rho_{y_0,t_0} \,d\mathrm{vol}_{\wt G(t)}=1.
\end{equation}
It now follows from the proof of White's regularity Theorem \cite[Theorem 3.5]{White2005} as in \cite{Wang2002} using \eqref{monot-Gauss} that $(y_0,t_0)$ is a regular point.

It remains to prove Claim~\ref{claim:gas-bdd}. Since $F_t$ is given by graph of $f_t:M\to N$, we might assume $F(x,t)=(x,f_t(x))\in M\times N$ embedded in $\mathbb{R}^Q$. Since $X$ is isometrically embedded into $\mathbb{R}^Q$ and $X$ is compact,  there exists $C_4>0$ such that
$$C_4^{-1} d_{euc}(p,q)\leq d_{G(t_0)}(p,q)\leq C_4 d_{euc}(p,q)$$
for all $p,q\in X\subset \mathbb{R}^Q$. Using also the fact that $G(t)$ are uniformly equivalent to $G(t_0)$ on $X$, we see that for all $p,q\in X\subset \mathbb{R}^Q$,
$$C_4^{-1} d_{euc}(p,q)\leq d_{G(t)}(p,q)\leq C_4 d_{euc}(p,q).$$

Since $G(t)=\pi_M^*g(t)\oplus \pi_N^* h(t)$, we have
\begin{equation}
\left(d_{\pi_M^*g(t)}(p,q)\right)^2\leq \left(d_{G(t)}(p,q) \right)^2\leq  C_4 \left(d_{euc}(p,q)\right)^2.
\end{equation}

Furthermore on $M$, since $F_t=\mathrm{Id}\times f_t$ and $\lambda_i\leq C_1$,
\begin{equation}
g_{ij}\leq \eta_{ij}=g_{ij}+f_i^\a f_j^\b h_{\a\b}\leq (1+C_1^2)g_{ij}.
\end{equation}
Hence $\eta(t)$ is uniformly equivalent to $g(t)$ on $M\times [0,t_0)$.  Therefore,
\begin{equation}
\begin{split}
&\quad \int_{M} (t_0-t)^{-m/2}\exp\left(-\frac{d^2_{euc}(F(x,t),y_0)}{\a(t_0-t)}  \right)\,d\mathrm{vol}_{\eta(t)}(x)\\
&\leq C_5 \int_{M_t} (t_0-t)^{-m/2}\exp\left(-\frac{d^2_{G(t)}(y,y_0)}{\a(t_0-t)}  \right)\,d\mathrm{vol}_{\pi_M^*g(t)}(y)\\
&\leq C_5\int_{M}(t_0-t)^{-m/2}\exp\left(-\frac{d^2_{g(t)}(x,x_0)}{\a(t_0-t)}  \right)\,d\mathrm{vol}_{g(t)}(x).
\end{split}
\end{equation}

Since $g(t)$ has bounded geometry of infinity order,  the integral is bounded by  constant depending a constant independent of $t$. This proves  Claim~\ref{claim:gas-bdd}.
\end{proof}

\section{ Proof of the Main Theorems}

In this section, will prove our main results. We will continue to work under the setting in Section~\ref{sec:graph-geom}, and  show that under certain assumption on $G$, both the area non-increasing and strictly area decreasing will be preserved under the mean curvature flow coupled with the evolving metric $G(t)$.

\subsection{Evolution equations under Uhlenbeck trick}\label{subsec:Uhlenbeck}
We want to study the evolution of $S$ and $\Theta$. To simplify computation, we use the abstract vector bundle method as in \cite{TsaiTsuiWang2023} which we refer it as  Uhlenbeck trick. Let $x_0\in M$ and $t_0\in (0,T]$ be fixed. Let $\{E_A\}_{A=1}^m$ be an orthonormal frame with respect to $\eta(t_0)$ at $x_0\in M$ so that $dF_{t_0}(E_i)=e_i$ given by \eqref{frame-onGraph} at $\left(x_1,f_{t_0}(x_1)\right)\in M\times N$ where $x_1=\phi_{t_0}(x_0)$. We extend $E_A$ around $x_0\in M$ by parallel transport with respect to $\eta(t_0)$, so that $\nabla E_A=\Delta E_A=0$ at $(x_0,t_0)$ and $\{dF_t(E_A)\}_{A=1}^m$ is an orthonormal basis for $t=t_0$. Consider the O.D.E:
\begin{equation}\label{eqn:Uhlenbeck}
\left\{
\begin{array}{ll}
\partial_t E_A^k=-\frac12 \eta^{jk} F^\a_i F^\b_j J_{\a\b} E_A^i -\eta^{jk}  H^\a A^\b_{ij}G_{\a\b} E_A^i ; \\
E_A^k(x,t_0)=E_A^k(x)
\end{array}
\right.
\end{equation}
where $J=\partial_t G$. Here we use the notation $H=H^\a \partial_\a$ and $A_{ij}^\a=F^\a_{|ij}$ to denote the mean curvature vector and the second fundamental form of $F_t(M)$ in $(M\times N,G(t))$. By our choice of endomorphism, direct computation shows that for all $t\in (0,T]$,
\begin{equation}
\eta(E_A,E_B)= \delta_{AB}.
\end{equation}

Now we are ready to derive evolution equations for $S$ and $\Theta$ using the graphical frame \eqref{frame-onGraph} and the Uhlenbeck trick. More precisely, at each $(x_0,t_0)\in M\times (0,T]$, we choose the graphical frame so that \eqref{eqn:graph-frame-1} holds at $(x_0,t_0)$ and then we extend it locally around $(x_0,t_0)$ in space-time using the discussion above to obtain $\{E_i(x,t)\}_{i=1}^m$ nearby $(x_0,t_0)\in M\times (0,T]$.

In the following, we will write $S_{ii}=S(E_i,E_i)=\frac{1-\lambda_i^2}{1+\lambda_i^2}$ and its conjugate $C_{ii}=\frac{2\lambda_i}{1+\lambda_i^2}$.  We also denote $K^g_{ip}$ to be the sectional curvature of the two plane spanned by $u_i, u_p$, etc. $K^g_{pp}=0$ by convention. In the setting of graph frame  \eqref{frame-onGraph} at a point, we always assume that $\lambda_1\ge\lambda_2\ge\cdots\ge\lambda_m\geq 0$. Hence we have
\be\label{e-S-min}
\left\{
  \begin{array}{ll}
    S_{11}\le S_{22}\le\cdots\le S_{mm};\\[2mm]
        C_{11}\ge C_{22}\ge\cdots\ge C_{mm};\\[2mm]
\Theta_{1221}\le \Theta_{ijji} \;\;\hbox{for all $i\neq j$}.
  \end{array}
\right.
\ee
 For any $t\ge0$, we also define
 \be\label{e-mt}
\mathfrak{m}(t)=\inf_{x\in M}\{\text{smallest eigenvalue of $\Theta(x,t)$}\}.
\ee

With the frame $\{E_i(x,t)\}_{i=1}^m$ around $(x_0,t_0)$, we might treat $S(E_i,E_j)$ as a locally defined function.  We have  the following evolution equation of $S$.
\begin{lma}\label{lma:evo-S}
For any $(x_0,t_0)\in M\times (0,T]$, under the  graphical frame $\{E_i\}_{i=1}^m$ we have
\begin{equation}
\begin{split}
\heat S(E_i,E_i)\Big|_{(x_0,t_0)}&=\mathbf{I}+\mathbf{II}+\mathbf{III}
\end{split}
\end{equation}
where
\begin{equation}
\left\{
\begin{array}{ll}
\mathbf{I}&\displaystyle=\sum_{a=1}^n\sum_{l=1}^m 2(S_{ii}+S_{aa}) |A^{a+m}_{il}|^2;\\
\mathbf{II}&\displaystyle=C_{ii}^2\cdot  \sum_{k=1}^m \frac{K^g_{ik}-\lambda_k^2 K^h_{ik}}{1+\lambda_k^2};\\
\mathbf{III}&=\displaystyle \frac12 C_{ii}^2 \cdot \left( \partial_t g_{ii}-\partial_t h_{ii}\right).
\end{array}
\right.
\end{equation}
Here $\lambda_k^2 K^h_{ik}$ is understood to be zero if $k>n$.
\end{lma}
\begin{proof}
It is a slight modification of  \cite[(3.7)]{TsuiWang2004}. Since we are working on an evolving background, we include the proof for readers' convenience. We start with the evolution equation of the tensor $S$. Firstly using $\partial_t F=H=\tau(F)$ and $s_{\a\b}=g_{\a\b}-h_{\a\b}$,  we have
\begin{equation}
\begin{split}
\partial_t S_{ij}&=D_t (F^\a_i F^\b_j s_{\a\b})\\
&=\nabla_i \Delta_\eta F^\a \cdot F^\b_j \tilde S_{\a\b}+\nabla_j \Delta_\eta F^\b \cdot F^\a_i \tilde S_{\a\b}+F^\a_i F^\b_j \left(\partial_t g_{\a\b}-\partial_t h_{\a\b} \right)
\end{split}
\end{equation}
while
\begin{equation}
\begin{split}
\Delta_\eta S_{ij}&= \Delta_\eta F^\a_i \cdot F^\b_j \tilde S_{\a\b}+ \Delta_\eta F^\b_j \cdot F^\a_i \tilde S_{\a\b}+2F^\a_{|ik} F^\b_{|jl}\eta^{kl} \tilde S_{\a\b}.
\end{split}
\end{equation}

Now we apply the Ricci identity of $(M,\eta)$ so that
\begin{equation}
\begin{split}
 \Delta_\eta F^\a_i&= \nabla_i \Delta_\eta F^\a+ R_i^p F^\a_p -\eta^{kl}\tilde R_{\delta\gamma\e}\,^\a F^\gamma_k F^\delta_i F^\e_l
\end{split}
\end{equation}
where $R$ and $\tilde R$ denote the curvature of $\eta$ and $G$ respectively. On the other hand, Gauss equation infers that
\begin{equation}
\begin{split}
R_{iq}&=\eta^{kl} R_{iklq}=\eta^{kl} \left( \tilde R_{\delta\gamma\e\sigma} F_i^\delta F_k^\gamma F_l^\e F^\sigma_q- A_{il}^\delta A_{kq}^\gamma G_{\delta\gamma}+A_{iq}^\delta A_{kl}^\gamma G_{\delta\gamma}  \right)
\end{split}
\end{equation}

Combining all, we arrive at
\begin{equation}\label{eqn:evo-S-tensor}
\begin{split}
\heat S_{ij}&=\eta^{kl}   S_{j}^qA_{il}^\delta A_{kq}^\gamma G_{\delta\gamma}   +\eta^{kl} S_{i}^q A_{jl}^\delta A_{kq}^\gamma G_{\delta\gamma}\\
&\quad - S_{j}^q A_{iq}^\delta H^\gamma G_{\delta\gamma} - S_{i}^q A_{jq}^\delta H^\gamma G_{\delta\gamma}   \\
&\quad+ (G^{\a\sigma} -\eta^{pq}  F^\a_p F^\sigma_q )  \eta^{kl}  F^\b_jF_i^\delta F_k^\gamma F_l^\e  s_{\a\b}   \tilde R_{\delta\gamma\e\sigma} \\
&\quad + (G^{\b\sigma} -\eta^{pq}  F^\b_p F^\sigma_q )  \eta^{kl}  F^\a_iF_j^\delta F_k^\gamma F_l^\e  s_{\a\b}   \tilde R_{\delta\gamma\e\sigma} \\
&\quad -2A^\a_{ik} A^\b_{jl}\eta^{kl} s_{\a\b} +F^\a_i F^\b_j ( \partial_t g_{\a\b}-\partial_t h_{\a\b}).
\end{split}
\end{equation}

Now we employ the bundle trick by considering $S(E_i,E_i)$ for $1\leq i\leq m$. Using \eqref{eqn:graph-frame-1}, \eqref{eqn:graph-frame-2}, \eqref{eqn:Uhlenbeck} and \eqref{eqn:evo-S-tensor}, at $(x_0,t_0)$ it satisfies
\begin{equation}\label{eqn:evo-S-frame}
\begin{split}
&\quad \heat  S(E_i,E_i)\\
&=\left[\heat S \right] (E_i,E_i)+2S(\partial_t E_i,E_i)\\
&=\left[\frac{2(1-\lambda_i^2)}{1+\lambda_i^2} |A^{a+m}_{il}|^2+\frac{2(1-\lambda_a^2)}{1+\lambda_a^2}|A^{a+m}_{il}|^2\right]\\
&\quad -\left[\frac{2\lambda_i}{1+\lambda_i^2}\tilde R(\tilde E_i,\tilde E_k,\tilde E_k,\tilde E_{m+i})+\frac{2\lambda_i}{1+\lambda_i^2} \tilde R(\tilde E_i,\tilde E_k,\tilde E_k,\tilde E_{m+i})\right]\\
&\quad  +\left[(\partial_t g-\partial_t h)(\tilde E_i,\tilde E_i)-S_{ii}\cdot \left(\partial_t g+\partial_t h \right)(\tilde E_i,\tilde E_i)\right]\\
&=\mathbf{I}+\mathbf{II}+\mathbf{III}.
\end{split}
\end{equation}

For $\mathbf{II}$, using the product structure of $G(t)$,  we have
\begin{equation}
\begin{split}
 \tilde R(\tilde E_i,\tilde E_k,\tilde E_k,\tilde E_{m+i})
&=\frac{\tilde R(u_i+\lambda_i v_i,u_k+\lambda_k v_k,u_k+\lambda_k v_k, -\lambda_i u_i+v_i)}{(1+\lambda_k^2)(1+\lambda_i^2)}\\
&=\frac{-\lambda_i R^g(u_i,u_k,u_k,u_i)+ \lambda_i \lambda_k^2 R^h(v_i,v_k,v_k,v_i)}{(1+\lambda_k^2)(1+\lambda_i^2)}\\
&=-\frac{\lambda_i}{1+\lambda_i^2}\cdot \frac{K^g_{ik}-  \lambda_k^2 K^h_{ik}}{1+\lambda_k^2}.
\end{split}
\end{equation}
The assertion on $\mathbf{II}$ follows since $C_{ii}=\frac{2\lambda_i}{1+\lambda_i^2}$.  $\mathbf{III}$ is similar.
\end{proof}

Similarly, $\Theta(E_i,E_j,E_k,E_l)$ is a locally defined function around $(x_0,t_0)$.  For notation convenience, we just write
$$\Theta_{ijkl}=\Theta(E_i,E_j,E_k,E_l)$$
defined locally around $(x_0,t_0)$.  The following is a slight modification of \cite[Lemma 3.1]{TsaiTsuiWang2023}.
\begin{lma}\label{l-positivity} For any $(x_0,t_0)\in M\times (0,T]$, under the  graphical frame $\{E_i\}_{i=1}^m$,  if $\Theta_{1221}+\a>0$ at $(x_0,t_0)$ for some $\a\in \mathbb{R}$, then
\bee
\begin{split}
& (\Theta_{1221}+\a)\heat \Theta_{1221}+\frac12 |\nabla \Theta_{1221}|^2\\
\ge&-2\a(\Theta_{1221}+\a)|A|^2+  4\a  \lf(S_{11}\sum_{k=1}^m |A^{1+m}_{1k}|^2+S_{22}\sum_{k=1}^m |A^{2+m}_{2k}|^2\ri)\\
&+(\Theta_{1221}+\a)\lf[(\mathbf{2}) +(\mathbf{3})\ri]
\end{split}
\eee
where $|A|$ is the norm of the second fundamental form and
\bee
 \left\{
   \begin{array}{ll}
(\mathbf{2}) =&\displaystyle C_{11}^2\sum_{k=1}^m \frac{K^g_{1k}-\lambda_k^2 K^h_{1k}}{1+\lambda_k^2}+C_{22}^2\sum_{k=1}^m \frac{K^g_{2k}-\lambda_k^2 K^h_{2k}}{1+\lambda_k^2}\\
(\mathbf{3})=&\displaystyle \frac12 C_{11}^2 \lf(\p_t g_{11}-\p_t h_{11}\ri)+ \frac12 C_{22}^2\lf(\p_t g_{22}-\p_t h_{22}\ri).
   \end{array}
 \right.
\eee
Here $\lambda_k^2 K^h_{ik}$ is understood to be zero if $k>n$.

\end{lma}
\begin{proof} By  In the following, let $E_i$ be the graph frame at a point, which has been extended to a frame using Uhlenbeck's trick, we write $S_{ii}=S(E_i,E_i)$ and $\Theta_{ijji}=\Theta(E_i,E_j,E_j,E_i)$ etc.  By Lemma \ref{lma:evo-S}, at $(x_0,t_0)$ we have
\be\label{eqn:Theta1221}
\begin{split}
  \heat \Theta_{1221}=&
 \heat\lf(S_{11}+S_{22}\ri)\\
 =&(\mathbf{1})+ (\mathbf{2})+(\mathbf{3})\\
 \end{split}
 \ee
 where
\bee
 \left\{
   \begin{array}{ll}
   (\mathbf{1}) =&\displaystyle\sum_{a=1}^n\sum_{l=1}^m 2(S_{11}+S_{aa})|A^{a+m}_{1l}|^2 +\sum_{a=1}^n\sum_{l=1}^m 2(S_{22}+S_{aa})|A^{a+m}_{2l}|^2\\[3mm]
(\mathbf{2}) =&\displaystyle C_{11}^2\sum_{k=1}^m \frac{K^g_{1k}-\lambda_k^2 K^h_{1k}}{1+\lambda_k^2}+C_{22}^2\sum_{k=1}^m \frac{K^g_{2k}-\lambda_k^2 K^h_{2k}}{1+\lambda_k^2}\\[3mm]
(\mathbf{3})=&\displaystyle \frac12 C_{11}^2 \lf(\p_t g_{11}-\p_t h_{11}\ri)+\frac12 C_{22}^2\lf(\p_t g_{22}-\p_t h_{22}\ri).
   \end{array}
 \right.
\eee

\bigskip

Since $S_{ii}+S_{jj}\ge \Theta_{1221}$ for all $1\leq i\neq j\leq m$,  $(\mathbf{1})$ can be estimated as:
\bee
\begin{split}
(\mathbf{1})=&\sum_{1\le k\le m;  a\neq 1, 2}2(S_{11}+S_{aa})|A_{1k}^{a+m}|^2
+  \sum_{k=1}^m 2(S_{11}+S_{22})|A_{1k}^{2+m}|^2
\\
&+\sum_{1\le k\le m; a\neq 1, 2}2(S_{22}+S_{aa})|A_{2k}^{a+m}|^2
+  \sum_{k=1}^m 2(S_{11}+S_{22})|A_{2k}^{1+m}|^2\\
&+4S_{11}\sum_{k=1}^m|A_{1k}^{1+m}|^2+4S_{22}\sum_{k=1}^m|A_{2k}^{2+m}|^2\\
\ge  &2\Theta_{1221}\lf( \sum_{1\le k\le m; a\neq 1, 2}(|A_{1k}^{a+m}|^2+|A_{2k}^{a+m}|^2)
 +\sum_{k=1}^m (|A_{1k}^{2+m}|^2+ |A_{2k}^{1+m}|^2)\ri)\\
&+4S_{11}\sum_{k=1}^m|A_{1k}^{1+m}|^2+4S_{22}\sum_{k=1}^m|A_{2k}^{2+m}|^2.
\end{split}
\eee

On the other hand, for each $1\leq k\leq m$
\bee
\begin{split}
\nabla_{E_k} \Theta(E_1,E_2,E_2,E_1)=&-4\lf(\frac{A^{1+m}_{1k}\lambda_1}{1+\lambda_1^2}
+\frac{A^{2+m}_{2k}\lambda_2}{1+\lambda_2^2}\ri)\\
=&-2(C_{11}A^{1+m}_{1k}+C_{22}A^{2+m}_{2k})
\end{split}
\eee
so that
\bee
|\nabla \Theta_{1221}|^2=4\sum_{k=1}^m(C_{11}A^1_{1k}+C_{22}A^2_{2k})^2.
\eee

On the other hand, by the proof of \cite[Lemma 3.1]{TsaiTsuiWang2023},
$$
2\Theta_{ijji}S_{ii}=\Theta_{ijji}^2+C_{jj}^2-C_{ii}^2
$$
for all $i\neq j$. In fact,
\bee
\begin{split}
\Theta_{ijji}^2+C_{jj}^2-C_{ii}^2=&\Theta_{ijji}^2
+\frac{2(\lambda_j^2-\lambda_i^2)\Theta_{ijji}}{(1+\lambda_i^2)(1+\lambda_j^2)}\\
=&\Theta_{ijji}\cdot \frac{2(1-\lambda_i^2\lambda_j^2+\lambda_j^2-\lambda_i^2)} {(1+\lambda_i^2)(1+\lambda_j^2)}\\
=&2\Theta_{ijji}S_{ii}.
\end{split}
\eee
Hence we have
\bee
\begin{split}
&  \sum_{k=1}^m \Theta_{1221}\lf(4S_{11}(A_{1k}^{1+m})^2+4S_{22}(A_{2k}^{2+m})^2\ri)
 +\frac12|\nabla\Theta_{1221}|^2\\
=&2\sum_{k=1}^m |A^{1+m}_{1k}|^2\lf(\Theta_{1221}^2+C_{22}^2-C_{11}^2\ri)\\
&+2\sum_{k=1}^m |A^{2+m}_{2k}|^2\lf(\Theta_{1221}^2+C_{11}^2-C_{22}^2\ri)
+2\sum_{k=1}^m(C_{11}A^{1+m}_{1k}+C_{22}
A^{2+m}_{2k})^2\\
= &2\Theta_{1221}^2\sum_{k=1}^m\lf(|A^{1+m}_{1k}|^2+|A^{2+m}_{2k}|^2\ri)
+2\sum_{k=1}^m(C_{22}A^{1+m}_{1k}+C_{11}A^{2+m}_{2k})^2\\
\ge&2\Theta_{1221}^2\sum_{k=1}^m\lf(|A^{1+m}_{1k}|^2+|A^{2+m}_{2k}|^2\ri).
\end{split}
\eee

Therefore,
\bee
\begin{split}
&\quad  (\Theta_{1221}+\a) (\mathbf{1})+\frac12 |\nabla \Theta_{1221}|^2\\
&\geq  2\Theta_{1221} (\Theta_{1221}+\a) \lf(\sum_{1\le k\le m; a\neq 1, 2}(|A_{1k}^{a+m}|^2+|A_{2k}^{a+m}|^2)
 +\sum_{k=1}^m (|A_{1k}^{2+m}|^2+ |A_{2k}^{1+m}|^2)\ri)\\
&+(\Theta_{1221}+\a)\left[4S_{11}\sum_{k=1}^m(A_{1k}^{1+m})^2+4S_{22}\sum_{k=1}^m(A_{2k}^{2+m})^2\right]+\frac12 |\nabla \Theta_{1221}|^2\\
&\ge 2\Theta_{1221} (\Theta_{1221}+\a) \lf(\sum_{1\le k\le m; a\neq 1, 2}(|A_{1k}^{a+m}|^2+|A_{2k}^{a+m}|^2)
 +\sum_{k=1}^m (|A_{1k}^{2+m}|^2+ |A_{2k}^{1+m}|^2)\ri)\\
 &+2\Theta_{1221}^2\sum_{k=1}^m\lf(|A^{1+m}_{1k}|^2+|A^{2+m}_{2k}|^2\ri)
   +4\a \left[S_{11}\sum_{k=1}^m(A_{1k}^{1+m})^2+S_{22}\sum_{k=1}^m(A_{2k}^{2+m})^2\right]\\
&\geq -2\a (\Theta_{1221}+\a) \sum_{k=1}^m\sum_{a=1}^n(|A_{1k}^{a+m}|^2+|A_{2k}^{a+m}|^2)\\
&\quad +4\a \left[S_{11}\sum_{k=1}^m(A_{1k}^{1+m})^2+S_{22}\sum_{k=1}^m(A_{2k}^{2+m})^2\right].
\end{split}
\eee
because $(\Theta_{1221}+\a)\Theta_{1221}\ge -\a(\Theta_{1221}+\a)$. Putting this back to \eqref{eqn:Theta1221} yields the result.
\end{proof}
In case $\Theta>0$, we have the following which is obtained in \cite{TsaiTsuiWang2023} and has been used in the proof of long time existence in \S2.3.

\begin{cor}\label{c-det}
As in Lemma \ref{l-positivity}, if $\Theta>0$, then

\bee
\heat \log\lf(\frac{\det\Theta}{\det(\frac12\eta\owedge\eta)}\ri)\ge a|A|^2-C
\eee
for { some $a>0$ and} some constant $C$ depending only on the bounds of the curvatures of $g(t), h(t)$, the bounds of $\p_t g, \p_t h$,  $m, n$ and  the positive lower bound of $\Theta$. 
\end{cor}
\begin{proof} Since $\Theta>0$, we can take $\a=0$ in the proof of Lemma \ref{l-positivity} before applying $(\Theta_{1221}+\a)\Theta_{1221}\ge -\a(\Theta_{1221}+\a)$. The result follows.
\end{proof}

\subsection{Monotonicity in static background}

In this subsection, we consider the case when $G(t)$ is static in $t$ and will prove the preservation of area non-increasing under various curvature conditions. Let $(M^m,g), (N^n,h)$ be smooth compact manifolds. Let $\ell=\min\{m,n\}\ge 2$.
\vskip .2cm

({\bf A}): $g(t)=g$ and  $h(t)=h$ are time-independent satisfying
\be\label{e-A}
\left\{
\begin{array}{ll}
 \Ric^g_{\min}-\Ric_{\max}^h+(m-\ell)\cdot \kappa_M+(n-\ell)\cdot \kappa_N\ge0, \;\;{\text and\ }\\[3mm]
 \kappa_M+\kappa_N\ge 0
\end{array}
\right.
\ee
\vskip .2cm

({\bf B}): $g(t)=g$ and  $h(t)=h$ are time-independent satisfying
\be\label{e-B}
\kappa_M\ge 0\ \text{and}\ \tau_N\le \frac{2(m-\ell)+\ell-1}{\ell-1}\kappa_M
\ee
  Here  $\kappa_M$  is the lower bounds of the sectional curvature of $g$ and $\tau_N$ is the upper bound of the sectional  curvature of $h$.
\vskip .2cm

We want to estimate $(\mathbf{2})+(\mathbf{3})$ in Lemma~\ref{l-positivity} under the above conditions. Note that $(\textbf{3})$ is always zero in static case.
\begin{lma}\label{l-R-1}  With the same assumptions and notations as in Lemma \ref{l-positivity}, we have the following:
\begin{enumerate}
  \item [(i)] Under the condition ({\bf A}), we have
\bee
\begin{split}
(\mathbf{2})+(\mathbf{3})\ge & \frac12 \sum_{p=3}^\ell\lf[ C_{11}^2(K_{1p}^g+K_{1p}^h)+
 C_{22}^2(K_{2p}^g+K_{2p}^h)\ri]S_{pp}
\\&+(K^g_{12}+K^h_{12}) \frac{ (\lambda_1^2+\lambda_2^2) }{ (1+\lambda_1^2) (1+\lambda_2^2)} \Theta_{1221}.
\end{split}
\eee
  \item[(ii)] Under condition ({\bf B}), we have
  \bee
 (\mathbf{2})+(\mathbf{3})\ge(\kappa_M+\tau_N)\lf[\frac12\lf( C_{11}^2+C_{22}^2\ri)\sum_{p=3}^\ell S_{pp}+  \frac{\lambda_1^2+\lambda_2^2}{(1+\lambda_1^2)(1+\lambda_2^2)}\Theta_{1221}\ri].
  \eee
\end{enumerate}
\end{lma}
\begin{proof}
(i) Assume ({\bf A}). Then $({\bf 3})=0$ and
\bee
\begin{split}
(\mathbf{2}) =&C_{11}^2\sum_{p=1}^m \frac{K^g_{1p}-\lambda_p^2 K^h_{1p}}{1+\lambda_p^2}+C_{22}^2\sum_{p=1}^m \frac{K^g_{2p}-\lambda_p^2 K^h_{2p}}{1+\lambda_p^2}
\end{split}
\eee
where
\bee
\begin{split}
& 2\sum_{p=1}^m \frac{K^g_{1p}-\lambda_p^2 K^h_{1p}}{1+\lambda_p^2}\\
=&\sum_{p=1}^m K^g_{1p}(1+S_{pp}) -\sum_{p=1}^\ell K^h_{1p}(1-S_{pp})\\
{\ge} &\Ric^g_{11}-\Ric^h_{11}+\sum_{p=2}^\ell (K^g_{1p}+K^h_{1p})S_{pp}
+(m-\ell)\kappa_M+(n-\ell)\kappa_N\\
\ge & \sum_{p=3}^\ell(K^g_{1p}+K^h_{1p})S_{pp}+(K^g_{12}+K^h_{12})S_{22}.
\end{split}
\eee

Similarly,
\bee
\begin{split}
2\sum_{p=1}^m \frac{K^g_{2p}-\lambda_p^2 K^h_{2p}}{1+\lambda_p^2}
\ge & \sum_{p=3}^\ell(K^g_{2p}+K^h_{2p})S_{pp}+(K^g_{12}+K^h_{12})S_{11}.
\end{split}
\eee
On the other hand, since
 \bee
 \begin{split}
 C_{11}^2S_{22}+C_{22}^2S_{11}
 =&\frac{2 (\lambda_1^2+\lambda_2^2) }{ (1+\lambda_1^2) (1+\lambda_2^2)}\Theta_{1221},
 \end{split}
 \eee
we conclude that (i) is true.

\vskip0.3cm

(ii) Assume ({\bf B}). Observe that $\tau_N$ can be assumed to be nonnegative.
In this case,
\bee
\begin{split}
2\sum_{p=1}^m \frac{K^g_{1p}-\lambda_p^2 K^h_{1p}}{1+\lambda_p^2}
\ge & \sum_{p=2}^m\frac{2\kappa_M}{1+\lambda_p^2}-\tau_N\sum_{p=2}^n \frac{2\lambda_p^2}{1+\lambda_p^2}\\
=&\sum_{p=2}^m(1+S_{pp})\kappa_M-\tau_N\sum_{p=2}^\ell (1-S_{pp}) \\
=&\sum_{p=2}^\ell (\kappa_M+\tau_N)S_{pp}+(\ell-1)(\kappa_M -\tau_N)+2(m-\ell) \kappa_M\\
\ge&\sum_{p=3}^\ell (\kappa_M+\tau_N)S_{pp}+(\kappa_M+\tau_N)S_{22}
\end{split}
\eee
Similarly,
\bee
\begin{split}
2\sum_{p=1}^m \frac{K^g_{2p}-\lambda_p^2 K^h_{2p}}{1+\lambda_p^2}\ge &\sum_{p=3}^\ell (\kappa_M+\tau_N)S_{pp}+
 (\kappa_M+\tau_N)S_{11}
\end{split}
\eee
so that
\bee
\begin{split}
(\mathbf{2})\ge &\frac12\lf(C_{11}^2+C_{22}^2\ri)\sum_{p=3}^\ell (\kappa_M+\tau_N)S_{pp}\\
 &+\frac12(\kappa_M+\tau_N)\lf(C_{11}^2 S_{22}+C_{22}^2S_{11}\ri)\\
=&(\kappa_M+\tau_N)\lf[\frac12\lf( C_{11}^2+C_{22}^2\ri)\sum_{p=3}^\ell S_{pp}+  \frac{\lambda_1^2+\lambda_2^2}{(1+\lambda_1^2)(1+\lambda_2^2)}\Theta_{1221}\ri].
\end{split}
\eee
As before, one can conclude that (ii) is true.

\end{proof}

Recall that for  $t\ge 0$,
\bee
\mathfrak{m}(t)=\inf_{x\in M}\{\text{smallest eigenvalue of $\Theta(x,t)$}\}.
\eee
Now we are ready to prove the Case {\bf (A)} and {\bf (B)} in our main Theorem.
\begin{thm}\label{t-TTW-2}
 Let $(M^m, g), (N^n,h)$ be two compact manifolds. Suppose $f_0$ is a smooth map from $M$ to $N$. Let $F: M\to (M\times N, g\oplus h)$ be a smooth mean curvature flow defined on $M\times[0,T)$ with $0<T\leq +\infty$ and initial map $F_0=\mathrm{Id} \times f_0$. Moreover, assume $F$ is a graph given by a map $f_t:M\to N$ for all $t\in [0,T)$. Assume that $f_0:M\to N$ is area non-increasing and  one of the condition  ({\bf A}) or ({\bf B}) holds, then
 \begin{enumerate}
   \item [(i)] $f_t$ is area non-increasing for $t\in [0,T)$;
   \item [(ii)] If $f_0$ is strictly area decreasing at a point,  then $ \mathfrak{m}(t)>0$ for $t>0$ and is nondecreasing in $t$ where $\mathfrak{m}(t)$ is defined in \eqref{e-mt}. In particular, $F$ has long time solution. Moreover, if $\kappa_M+\kappa_N>0$ in condition ({\bf A}) or $\kappa_M >0$ in condition ({\bf B}), then $f_0$ is homotopically trivial.
\item[(iii)] If   $m, n\ge 3$, and $\kappa_M+\kappa_N>0$ in condition ({\bf A}) or $\kappa_M >0$ in condition ({\bf B}), then either $f_0$ homotopically trivial  or  $f_0$ is a  Riemanian submersion (if $m>n$), local isometry (if $m=n$),  isometric immersion (if $m<n$).
 \end{enumerate}
\end{thm}
Before we prove the theorem, we want to point out that the results on $f_0$ being homotopically trivial under the assumption that $f_0$ is strictly area decreasing have been obtained by Lee-Lee \cite{LeeLee2011} and Tsai-Tsui-Wang \cite{TsaiTsuiWang2023}. We just slightly generalize their results to assume that $f_0$  is strictly area decreasing at a point.

\begin{proof}[Proof of Theorem \ref{t-TTW-2}] We only prove the case when $(M,g), (N,h)$ satisfy condition ({\bf A}). The proof is similar if they satisfy condition ({\bf B}).  We may assume that $F$ is smooth on $M\times[0,T]$. To prove (i) and (ii), let $0\le \varphi_0\le 1$ be a smooth function on $M$ and let $\varphi$ be the solution of the heat equation:
\be\label{e-phi}
\left\{
  \begin{array}{ll}
\displaystyle \heat \varphi=0, & \hbox{on $M\times[0,T]$;}\\[3mm]
\varphi=\varphi_0& \hbox{at $t=0$.}
  \end{array}
\right.
\ee
By the maximum principle and the strong maximum principle, we have $1\ge\varphi\ge0$ and $  \varphi>0$ for $t>0$ if $\varphi_0$ is positive somewhere. Consider the following perturbation of $\Theta$:
\begin{equation}\label{eqn:perturbing-Theta}
\tilde\Theta=\Theta -  \frac12 \wt\varphi\cdot \eta\owedge \eta
\end{equation}
where $\e,L,\delta>0$, with $0<\e<1, 0<\delta<1$ and
\be\label{e-wtphi}
\wt \varphi=  \delta  e^{-Lt}\varphi^2-\e e^{L(t-T)}.
\ee
where $L>0$ is to be chosen.  Define
\be
\wt{\mathfrak{m}}(t)=\inf_{x\in M}\{\text{smallest eigenvalue of $\wt\Theta(x,t)$}\}.
\ee
Suppose $\wt{\mathfrak{m}}(0)>0$ and suppose $\wt{\mathfrak{m}}(t)<0$ for some $t>0$. Then there is $0<t_0\le T$ such that $\wt{\mathfrak{m}}(t_0)=0$ and $\wt{\mathfrak{m}}(t)>0$ for $ 0\le t< t_0$. Hence by \eqref{e-S-min}, there is $x_0\in M$ and in the graph frame at $(x_0,t_0)$,  $\wt\Theta(E_1,E_2,E_2,E_1)=0=\wt{\mathfrak{m}}(t_0)$. Extend $E_i$ using Uhlenbeck's trick to an open set in spacetime near $(x_0,t_0)$ following the discussion
in subsection~\ref{subsec:Uhlenbeck}. Then we have
\be\label{e-wtTheta-1}
\left\{
  \begin{array}{ll}
    \heat (\wt\Theta(E_1,E_2,E_2,E_1))\le 0;\\
 \nabla \Theta_{1221}-2\delta e^{-Lt_0}\varphi\nabla\varphi=\nabla \wt\Theta(E_1,E_2,E_2,E_1)=0.
  \end{array}
\right.
\ee

Here and below, we write $\wt\Theta(E_1,E_2,E_2,E_1)$ as $\wt\Theta_{1221}$ etc.
Since $\wt\Theta_{ippi}\ge \wt\Theta_{1221}$ for $i=1, 2$ and $p\ge 3$. Hence if $p\ge 3$, then $S_{pp}\ge \frac12\Theta_{1221}=\frac12(\wt\Theta_{1221}+\wt\varphi)$. Moreover, $ \Theta_{1221}=\wt\varphi$ at $(x_0,t_0)$.

 Let $\a>0$ to be determined later   such that $\Theta_{1221}+\a>0$.
by Lemma \ref{l-positivity}, Lemma \ref{l-R-1}, \eqref{e-wtTheta-1} implies:
\be \label{e-max-wtTheta-1}
\begin{split}
&\frac12|\nabla\Theta_{1221}|^2\\
\ge &(\Theta_{1221}+\a)\heat  \wt\Theta_{1221}+\frac12|\nabla\Theta_{1221}|^2\\
= &(\Theta_{1221}+\a  )\lf[\heat   \Theta_{1221} + L(\wt\varphi+2\e e^{L(t_0-T)})+2\delta e^{-Lt_0}|\nabla\varphi|^2\ri] + \frac12|\nabla\Theta_{1221}|^2\\
\ge &-2\a(\Theta_{1221}+\a)  |A|^2 +4 \a  \lf(S_{11}\sum_{p=1}^m (A^{1+m}_{1p})^2+S_{22}\sum_{p=1}^m (A^{2+m}_{2p})^2\ri)\\
&+(\Theta_{1221}+\a) \bigg[\sum_{p=3}^\ell\lf( \frac{2\lambda_1^2}{(1+\lambda_1^2)^2}(K^g_{1p}+K^h_{1p})
+
\frac{2\lambda_2^2}{(1+\lambda_2^2)^2}(K^g_{2p}+K^h_{2p})
\ri)S_{pp}
\\&+ (K^g_{12}+K^h_{12}) \frac{ (\lambda_1^2+\lambda_2^2) }{ (1+\lambda_1^2) (1+\lambda_2^2)} \Theta_{1221}\bigg]\\
&+(\Theta_{1221}+\a) \lf( L(\wt\varphi+2\e e^{L(t_0-T)})+2\delta e^{-Lt_0}|\nabla\varphi|^2\ri).\\
\end{split}
\ee

Hence if we  let $\b=\delta  e^{-Lt_0}\varphi^2+\e e^{ L(t_0-T)}$ and $\a=4\b$, then $\Theta_{1221}=
\delta  e^{-Lt_0}\varphi^2-\e e^{ L(t_0-T)}$ and so $\Theta_{1221}+\a\ge 3\b>0$ { at $(x_0,t_0)$.} {And $\nabla \Theta_{1221}=2\delta e^{-Lt} \varphi \nabla \varphi$ at $(x_0,t_0)$}. Hence,
\be
\begin{split}
0\ge &-C_1\b^2 +16\b  \lf(S_{11}\sum_{p=1}^m (A^{1+m}_{1p})^2+S_{22}\sum_{p=1}^m (A^{2+m}_{2p})^2\ri)\\
&+(\Theta_{1221}+\a) \bigg[\sum_{p=3}^\ell\lf( \frac{2\lambda_1^2}{(1+\lambda_1^2)^2}(K^g_{1p}+K^h_{1p})
+
\frac{2\lambda_2^2}{(1+\lambda_2^2)^2}(K^g_{2p}+K^h_{2p})
\ri)S_{pp}
\\&+ (K^g_{12}+K^h_{12}) \frac{ (\lambda_1^2+\lambda_2^2) }{ (1+\lambda_1^2) (1+\lambda_2^2)} \Theta_{1221}\bigg]\\
&+3\b \lf( L\b +2\delta e^{-Lt_0}|\nabla\varphi|^2\ri)-2\delta^2e^{-2Lt_0}\varphi^2|\nabla\varphi|^2
\end{split}
\ee
where $|A|$ is the norm of the second fundamental form and $C_1$ is a constant depending only on the upper bounds of $|A|, \lambda_i, m, n$ which is independent of $L, \delta, \e$.   By \eqref{e-wtTheta-1}, we have for $1\le p\le \ell$,
\bee
\begin{split}
0=&\nabla_{E_p}\wt\Theta_{1221}\\
=&\nabla_{E_p}\Theta_{1221} - 2\delta e^{-Lt_0}\varphi\varphi_p\\
=&-4\lf(\frac{A^{1+m}_{1p}\lambda_1}{1+\lambda_1^2}+\frac{A^{2+m}_{2p}\lambda_2}{1+\lambda_2^2}\ri)
-2\delta\varphi e^{-Lt_0}\varphi_p.
\end{split}
\eee

{\bf Claim}: $\lambda_1^2\ge \frac13$. In fact, since $\lambda_1\ge\lambda_2$, we have
\bee
2S_{11}\le S_{11}+S_{22}=\Theta_{1221}=\wt\varphi\le \delta\le 1.
\eee
From this one can see that $\lambda_1^2\ge \frac13$ and hence
\bee
\begin{split}
| (A^{1+m}_{1p})^2-  (A^{2+m}_{2p})^2|\le &2|A| |A^{1+m}_{1p}  +  A^{2+m}_{2p} |\\
\le &C_2\lf(\delta\varphi e^{-Lt_0}|\nabla \varphi|+\lf|1-\frac{\lambda_2(1+\lambda_1^2)}{\lambda_1(1+\lambda_2^2)}\ri|\ri)\\
\le&C_2\lf(\delta\varphi e^{-Lt_0}|\nabla \varphi|+ \b \ri).
\end{split}
\eee
for some constant  $C_2>0$   depending only on the upper bounds of $|A|, \lambda_i$. Here we have used the {\bf Claim} that $\lambda_1^2\ge\frac13$ and the fact that
\bee
 \lf|1-\frac{\lambda_2(1+\lambda_1^2)}{\lambda_1(1+\lambda_2^2)}\ri|=
\lf|\frac{(1+\lambda_2^2)(\lambda_1-\lambda_2)(1-\lambda_1\lambda_2)}{\lambda_1(1+\lambda_2^2)}\ri|
\le C|\Theta_{1221}|\le C\b
\eee
for some constant depending only on the upper bound of $\lambda_i$. Hence
\bee
\begin{split}
S_{11}  (A^{1+m}_{1p})^2+S_{22}  (A^{2+m}_{2p})^2=&S_{11}\lf((A^{1+m}_{1p})^2- ((A^{2+m}_{2p})^2\ri)+\Theta_{1221}(A^{2+m}_{2p})^2\\
\ge&-C_3\lf(\delta\varphi e^{-Lt_0}|\nabla\varphi|+\b\ri)
\end{split}
\eee
for some constant $C_3$ depending only on the upper bounds of $|A|, \lambda_i$. Combining this with \eqref{e-max-wtTheta-1}, using the facts that $S_{pp}\ge \frac12\Theta_{1221}=-\frac12\wt\varphi\ge-\frac12\b$ for $p\ge 3$ and $\kappa_M+\kappa_N>0$, we have
\bee
\begin{split}
0\ge &-C_4\lf( \b^2+\b\delta e^{-Lt_0}\varphi|\nabla\varphi|\ri)+ 3\b \lf( L\b +2\delta e^{-Lt_0}|\nabla\varphi|^2\ri)-2\delta^2e^{-2Lt_0}\varphi^2|\nabla\varphi|^2
\end{split}
\eee
for some $C_4>0$ depending only on the upper bounds of $|A|, \lambda_i,m, n$ and the curvatures of $g, h$.
Now
\bee
6\b\delta e^{-Lt_0}|\nabla\varphi|^2\ge 6\delta^2  e^{-2Lt_0}\varphi^2  |\nabla\varphi|^2
\eee
and
\bee
C_4\b\delta e^{-Lt_0}\varphi|\nabla\varphi|\le   \delta^2e^{-2Lt_0} \varphi^2|\nabla\varphi|^2+\frac{C_4^2}{4}\b^2.
\eee
This implies that
\bee
0\ge -C_5\b^2+3L\b^2.
\eee
for some $C_5>0$  depending only on the upper bounds of $|A|, \lambda_i,m, n$ and the curvatures of $g, h$. This is a contradiction if we choose $L=C_5$. Namely for this choice of $L$, $\wt\Theta>0$ for $t>0$ provided $\wt\Theta>0$ at $t=0$.

Suppose $\Theta\ge 0$ initially, we let $\varphi_0=0$. Then the above result implies that $\Theta+\e e^{L(t-T)}\ge0$ for all $t>0$. Let $\e\to0$, we conclude that (i) is true.

To prove (ii), suppose $f_0$ is strictly area decreasing at some point $x$. Then we can find a smooth function $1\ge\varphi_0\ge0$ so that $\varphi_0>0$ at $x$ and  $\Theta-\frac12\varphi_0\eta\owedge\eta\ge0$ initially. Let $\varphi$ be the solution to \eqref{e-phi} and let $L=C_5$ be as above and $1>\delta,\e>0$. Let $\wt\varphi$ be as in \eqref{e-wtphi} Then we can conclude that
$\Theta- \frac12\wt\varphi\eta\owedge\eta\ge 0$ for all $t>0$. Let $\e\to 0$, since $\varphi>0$ at $t>0$ we conclude that $\Theta>0$ for $t>0$. This proves the first part of (ii).

In order to prove $ \mathfrak{m}(t)$ is non-decreasing, it is sufficient to prove that it is non-decreasing on $[t,T]$ for all $t>0$. Since $\Theta>0$ for $t>0$, without loss of generality, we may assume $\Theta\ge \rho_0>0$ at $t=0$ and to prove that $\Theta\ge\rho_0$ for $t>0$. For any $\e>0$ consider $\ol\Theta=\Theta+\frac12\e t\eta\owedge\eta$. We claim that the infimum of the eigenvalues of $\ol\Theta$ is attained at $t=0$. Otherwise, there is $t_0>0$ and $x_0\in M$ such that $\ol\Theta_{1221}$ attains the infimum in the graph frame. Using the extension with Uhlenbeck's trick, we have
$$
\heat\ol\Theta_{1221}\le 0; \ \nabla\Theta_{1221}=0
$$
at $(x_0,t_0)$. Since $\Theta>0$ everywhere, one obtain from as in \eqref{e-max-wtTheta-1} with $\a=0$,

\be \label{e-max-olTheta-1}
\begin{split}
0
 \ge &\heat  \ol\Theta_{1221}\\
\ge &   \bigg[\sum_{p=3}^\ell\lf( \frac{2\lambda_1^2}{(1+\lambda_1^2)^2}(K^g_{1p}+K^h_{1p})
+
\frac{2\lambda_2^2}{(1+\lambda_2^2)^2}(K^g_{2p}+K^h_{2p})
\ri)S_{pp}
\\&+ (K^g_{12}+K^h_{12}) \frac{ (\lambda_1^2+\lambda_2^2) }{ (1+\lambda_1^2) (1+\lambda_2^2)} \Theta_{1221}\bigg]+\e \\
\ge&\e
\end{split}
\ee
because $S_{pp}\ge\frac12\Theta_{1221}>0$ for $p\ge3$ and $\kappa_M+\kappa_N\ge0$. This is impossible. Let $\e\to0$, we conclude that $\mathfrak{m}(t)\ge \rho_0$ for $t\ge 0$. From this we also conclude that $F$ has long time solution which is a graph for all time by Theorem~\ref{thm:long-time-MCF-RF}.

Suppose $\kappa_M+\kappa_N>0$. Let $\wh\Theta=e^{-at}\Theta$ where $a>0$ to be determined. Suppose the infimum of the eigenvalues of  $\wh\Theta$ in $M\times[0,T]$ is attained at some spacetime point $(x_0,t_0)$ with $t_0>0$. Then as in \eqref{e-max-olTheta-1}, at this point,
\be \label{e-max-hatTheta-1}
\begin{split}
0
 \ge & \heat  \wh\Theta_{1221}\\
\ge &   (\kappa_M+\kappa_N)  \frac{ (\lambda_1^2+\lambda_2^2) }{ (1+\lambda_1^2) (1+\lambda_2^2)} \wh\Theta_{1221}-a\wh\Theta_{1221}.
\end{split}
\ee
Since $\Theta\ge \rho_0$ for all $t$, we conclude that $\lambda_i\le (\frac2{\rho_0}-1)$  in spacetime.  Since $\lambda_i^2$ are uniformly bounded in space and time, there is a constant $C_6>0$ depending only on $\rho_0$,
\bee
\begin{split}
\frac{ (\lambda_1^2+\lambda_2^2) }{ (1+\lambda_1^2) (1+\lambda_2^2)}\ge &C_6\lf(\frac{2\lambda_1^2}{1+\lambda_1^2}+\frac{2\lambda_2^2}{1+\lambda_2^2}\ri)\\
=&C_6(2-\Theta_{1221})\\
\ge&C_6(2- \mathfrak{m}_\infty).
\end{split}
\eee
 where $\mathfrak{m}_\infty=\lim_{t\to\infty}\mathfrak{m}(t)$ which exists and finite because $\mathfrak{m}(t)$ is nondecreasing and is bounded above by 2. We now argue by contradiction that $\mathfrak{m}_\infty=2$. Suppose $\mathfrak{m}_\infty<2$, then \eqref{e-max-hatTheta-1} implies:
\bee
0\ge C_7\Theta_{1221}-a\Theta_{1221}
\eee for some constant $C_7>0$ depending only on $C_6, \kappa_M+\kappa_N>0$ and $\mathfrak{m}_\infty<2$. Choose $a=\frac12 C_7$, we have a contradiction because $\Theta_{1221}>0$. Hence for this choice of $a$,  by letting $T\to+\infty$ we have $e^{-at}\mathfrak{m}(t)\ge \mathfrak{m}(0)\ge \rho_0>0$ for all $t$. This is impossible. Hence $\mathfrak{m}_\infty=2$ and $f_0$ is homotopically trivial.

To prove (iii), we claim the following: for $t>0$, suppose $\Theta_{1221}=0$ at a point, then at this point, $\lambda_i=1$ for all $1\le i\le \ell$. If this is true, then for $t>0$ either there is a point with $\Theta>0$ which implies $f_0$ is homotopically trivial, or $\lambda_i=1$ for $1\le i\le \ell$. Hence either $f_0$ is homotopically trivial, or if we let $t\to0$, we conclude that $\lambda_i=1$ for $1\le i\le \ell$ at $t=0$. This implies that $f_0$ is as described in the theorem.

To prove the claim, suppose $\Theta_{1221}=0$ at $(x_0,t_0)$ with $t_0>0$. By (i), $\Theta\ge0$ in spacetime, we have $\heat \Theta_{1221}=0, \nabla\Theta_{1221}=0$. Then by \eqref{e-max-wtTheta-1} for $\a>0$ with $\varphi=0, \e=0$, we have

\bee
\begin{split}
0 \ge &(\Theta_{1221}+\a)\heat   \Theta_{1221}+\frac12|\nabla\Theta_{1221}|^2\\
\ge &-2\a^2   |A|^2+4 \a  \lf(S_{11}\sum_{p=1}^m (A^{1+m}_{1p})^2+S_{22}\sum_{p=1}^m (A^{2+m}_{2p})^2\ri)\\
&+ \a    \sum_{p=3}^\ell\lf( \frac{2\lambda_1^2}{(1+\lambda_1^2)^2}(K^g_{1p}+K^h_{1p})
+
\frac{2\lambda_2^2}{(1+\lambda_2^2)^2}(K^g_{2p}+K^h_{2p})
\ri)S_{pp}
\end{split}
\eee
because $\Theta_{1221}=0$. Divide by $\a$ and then let $\a\to0$, we have

\be\label{e-rigidity-1}
\begin{split}
0 \ge &4  \lf(S_{11}\sum_{p=1}^m (A^{1+m}_{1p})^2+S_{22}\sum_{p=1}^m (A^{2+m}_{2p})^2\ri)\\
&+  (\kappa_M+\kappa_N)   \sum_{p=3}^\ell\lf( \frac{2\lambda_1^2}{(1+\lambda_1^2)^2}
+\frac{2\lambda_2^2}{(1+\lambda_2^2)^2}
\ri)S_{pp}
\end{split}
\ee
Since $\Theta_{1221}=0$ and $\lambda_1\ge\lambda_2$, we have $\lambda_1\ge 1$. Using the fact that $\nabla\Theta_{1221}=0$, as before,

\bee
\begin{split}
| (A^{1+m}_{1p})^2-  (A^{2+m}_{2p})^2|\le &2|A| |A^{1+m}_{1p}  -  A^{2+m}_{2p} |\\
\le &2|A|  |1-\frac{\lambda_2(1+\lambda_1^2)}{\lambda_1(1+\lambda_2^2)}| \\
=&0.
\end{split}
\eee
Hence
\bee
\begin{split}
\lf(S_{11}\sum_{p=1}^m (A^{1+m}_{1p})^2+S_{22}\sum_{p=1}^m (A^{2+m}_{2p})^2\ri)=&S_{11}\lf((A^{1+m}_{1p})^2-  (A^{2+m}_{2p})^2\ri)\\
=&0.
\end{split}
\eee
\eqref{e-rigidity-1} implies:
\bee
0 \ge     (\kappa_M+\kappa_N)   \sum_{p=3}^\ell\lf( \frac{2\lambda_1^2}{(1+\lambda_1^2)^2}+\frac{2\lambda_2^2}{(1+\lambda_2^2)^2}
\ri)S_{pp}.
\eee
 Since $S_{pp}\ge \frac12(S_{11}+S_{22})=0$, $\lambda_1\ge1$ and $\kappa_M+\kappa_N>0$, we have $\lambda_p=1$ for $p\ge 3$. Hence $\lambda_1\ge\lambda_2\ge\lambda_p\ge1$ and $1-\lambda_1^2\lambda_2^2=0$, we have $\lambda_p=1$ for all $1\leq p\leq \ell$. This completes the proof of the theorem under condition ({\bf A}). The case {(\bf B)} is similar.
\end{proof}

Here are some applications of Theorem \ref{t-TTW-2}:
\begin{cor}
\label{c-rigidity} Suppose $N$ is either $\mathbb{S}^n, n\ge 3$, $\mathbb{CP}^{n/2}, n\ge 4$ or $\mathbb{HP}^{n/4}, n\ge 4$ with the standard metrics. If $(M^n,g)$ is a compact manifold with nonnegative sectional curvature such that $\Ric_M\ge \Ric_N$. Then any area non-increasing map from $M$ to $N$ with nonzero degree must be an isometry.

\end{cor}
\begin{proof}  Since $f_0$ has nonzero degree, it is not homotopically trivial. By the theorem, we conclude that $f_0$ is a local isometry and hence is an isometry because $N$ is simply connected.
\end{proof}

\subsection{Monotonicity in evolving background}
 In this subsection, we consider the case when $G(t)$ is evolving.  This will eventually be applied to the case of non-negative $1$-isotropic curvature. We first introduce the following definition.

 \begin{defn} For a Riemannian manifold, at a point, $\Ric_3(\pi)$ is the Ricci curvature tensor restricted on a three dimensional subspace $\pi$ of the tangent space. We say that $\Ric_3\ge a$ at a point if $\Ric_3(\pi)\ge a$ for all
 $\pi$. We say that $\Ric_3\ge a$ on $M$ if it is true at all points. In this case, $a$ is called a lower bound of $\Ric_3$.
\end{defn}

We consider the following two situations:

({\bf C}): $\p_tg=-\Ric^g, \p_t h=-\Ric^h$ on $[0,T]$ where
\be
 \chi^g(t)+\chi^h(t)\ge 0; \ \ \text{and}\ \    (m-\ell) \cdot \chi^g(t)+(n-\ell)\cdot \chi^h(t)\ge0
\ee
for all $t\in [0,T]$.

\vskip .2cm
({\bf D}): $\p_tg=-\Ric^g, \p_t h=0$  on $[0,T]$ where
\be
 \chi^g(t) \ge 0; \ \ \text{and}\ \    \tau_N \le 0
\ee
for all $t\in [0,T]$.  Here $\chi^g(t),\chi^h(t)$ denote the infimum of $\Ric_3(g(t)),\Ric_3(h(t))$  and $\tau_N$ is the upper bound of the sectional curvature of $h$.

\begin{lma}\label{l-PIC-1}
With the same assumptions and notations as in Lemma \ref{l-positivity}, we have the following:
\begin{enumerate}
  \item [(i)] Under condition ({\bf C}),
 \bee
(\mathbf{2})+(\mathbf{3})\ge -C|\Theta_{1221}|+ {\frac{2\lambda_1^2}{(1+\lambda_1^2)^2}}
\sum_{p=3}^\ell(K^g_{1p}+K^g_{2p}  +K^h_{1p}+K^h_{2p}
  )S_{pp};
\eee

\item [(ii)] Under condition ({\bf D}),
\bee
(\mathbf{2})+(\mathbf{3})\ge -C|\Theta_{1221}|+  {\frac{2\lambda_1^2}{(1+\lambda_1^2)^2}}
\lf(\sum_{p=3}^\ell(K^g_{1p}+K^g_{2p})S_{pp}-\sum_{a=1}^\ell\frac{2\tau_N\lambda_a^2}{1+\lambda_a^2}\ri)
\eee
\end{enumerate}
for some constant $C>0$ depending only on the bounds of the curvatures of $g(t), h(t), m, n$.
 \end{lma}
 \begin{proof}
 (i) Assume {(\bf C)}. Then
\bee
\begin{split}
(\mathbf{2})+(\mathbf{3}) =
&\frac{2\lambda_1^2}{(1+\lambda_1^2)^2}\lf(\sum_{p=1}^m  \frac{2(K^g_{1p}-\lambda_p^2 K^h_{1p})}{1+\lambda_p^2}-\Ric^g_{11}+\Ric^h_{11}\ri)\\
&+\frac{2\lambda_2^2}{(1+\lambda_2)^2}\lf( \sum_{p=1}^m
\frac{2(K^g_{2p}-\lambda_p^2 K^h_{2p})}{1+\lambda_p^2}-\Ric^g_{22}+\Ric^h_{22}\ri)\\
\end{split}
\eee
where
\bee
\begin{split}
&\sum_{p=1}^m  \frac{2(K^g_{1p}-\lambda_p^2 K^h_{1p})}{1+\lambda_p^2}-\Ric^g_{11}+\Ric^h_{11}\\
=&\sum_{p=1}^m K^g_{1p}(1+S_{pp}) -\sum_{p=1}^\ell K^h_{1p}(1-S_{pp})-\sum_{p=1}^m K^g_{1p}+\sum_{p=1}^n K^h_{1p}\\
=& (K^g_{12}+K^h_{12})S_{22}+\sum_{p=3}^\ell(K^g_{1p}+K^h_{1p})S_{pp}
+\sum_{p=\ell+1}^m K^g_{1p}   +\sum_{p=\ell+1}^n K^h_{1p} .
\end{split}
\eee
Similarly,
\bee
\begin{split}
&\sum_{p=1}^m  \frac{2(K^g_{2p}-\lambda_p^2 K^h_{2p})}{1+\lambda_p^2}-\Ric^g_{22}+\Ric^h_{22}\\
=& (K^g_{12}+K^h_{12})S_{11}+\sum_{p=3}^\ell(K^g_{2p}+K^h_{2p})S_{pp}
+\sum_{p=\ell+1}^m K^g_{2p} +\sum_{p=\ell+1}^n K^h_{2p} \\
\end{split}
\eee
On the other hand,
\bee
\frac{2\lambda_1^2}{(1+\lambda_1^2)^2}-\frac{2\lambda_2^2}{(1+\lambda_2^2)^2}=
\frac{(\lambda_1^2-\lambda_2^2)\Theta_{1221}}{(1+\lambda_1^2)(1+\lambda_2^2)}.
\eee
Hence,
\bee
\begin{split}
(\mathbf{2})+(\mathbf{3}) =
&\lf(\frac{2\lambda_2^2}{(1+\lambda_2^2)^2}-\frac{2\lambda_1^2}{(1+\lambda_1^2)^2}\ri)\times\\
&\bigg[(K^g_{12}+K^h_{12})S_{11}+\sum_{p=3}^\ell(K^g_{2p}+K^h_{2p})S_{pp}
+\sum_{p=\ell+1}^m K^g_{1p} +\sum_{p=\ell+1}^n K^h_{1p}  \bigg]\\
&+\frac{2\lambda_1^2}{(1+\lambda_1^2)^2}\bigg[(K^g_{12}+K^h_{12})\Theta_{1221}
+\sum_{p=3}^\ell(K^g_{1p}+K^g_{2p}  +K^h_{1p}+K^h_{2p}
  )S_{pp}\\
  &+\sum_{p=\ell+1}^m (K^g_{1p}+K^g_{2p}) +\sum_{p=\ell+1}^n (K^h_{1p}+K^h_{2p})\bigg]\\
  \ge&-C|\Theta_{1221}| + {\frac{2\lambda_1^2}{(1+\lambda_1^2)^2}}  \sum_{p=3}^\ell(K^g_{1p}+K^g_{2p}  +K^h_{1p}+K^h_{2p}
  )S_{pp}
\end{split}
\eee
for some positive constants $C$  depending only on the bounds of the curvatures of $g(t), h(t), m, n$. This completes the proof of (i).

(ii) Assume ({\bf D}),  then
\bee
\begin{split}
(\mathbf{2})+(\mathbf{3}) =
&\frac{2\lambda_1^2}{(1+\lambda_1^2)^2}\lf(\sum_{p=1}^m  \frac{2(K^g_{1p}-\lambda_p^2 K^h_{1p})}{1+\lambda_p^2}-\Ric^g_{11} \ri)\\
&+\frac{2\lambda_2^2}{(1+\lambda_2)^2}\lf( \sum_{p=1}^m
\frac{2(K^g_{2p}-\lambda_p^2 K^h_{2p})}{1+\lambda_p^2}-\Ric^g_{22} \ri)\\
\end{split}
\eee
Now
\bee
\begin{split}
&\sum_{p=1}^m  \frac{2(K^g_{1p}-\lambda_p^2 K^h_{1p})}{1+\lambda_p^2}-\Ric^g_{11}\\
\ge&\sum_{p=1}^m K^g_{1p}(1+S_{pp}) -\sum_{a=1}^\ell\frac{2\tau_N\lambda_a^2}{1+\lambda_a^2} -\Ric^g_{11} \\
\geq & K^g_{12}S_{22}+\sum_{p=3}^\ell K^g_{1p}S_{pp}
 -\sum_{a=1}^\ell\frac{2\tau_N\lambda_a^2}{1+\lambda_a^2}  \\
\end{split}
\eee
Similarly,
\bee
\begin{split}
&\sum_{p=1}^m  \frac{2(K^g_{2p}-\lambda_p^2 K^h_{2p})}{1+\lambda_p^2}-\Ric^g_{22} \\
\ge& K^g_{12}S_{11}+\sum_{p=3}^\ell K^g_{2p}S_{pp}
 -\sum_{a=1}^\ell\frac{2\tau_N\lambda_a^2}{1+\lambda_a^2}.
\end{split}
\eee
Hence as in the proof of (i), one can conclude that (ii) is true.
\end{proof}

We first show that the area non-increasing is preserved as long as the ambient space is evolving and the correspond Rigidity. More precisely, we have the following.
 \begin{thm}\label{t-pic}
Let $(M^m, g(t)), (N^n,h(t))$ be two compact manifolds with $\ell=\min\{m,n\}\geq 3$, $t\in [0,T]$.    Suppose $f_0$ is a smooth map from $M$ to $N$. Let $F: M\to (M\times N, g\oplus h)$ be a smooth mean curvature flow defined on $M\times[0,T]$   with initial map $F_0=\mathrm{Id}\times f_0$. Moreover, assume $F$ is a graph given by a map $f_t:M\to N$ for all $t$. Assume that $f_0$ is with area non-increasing from $(M,g(0))$ to $(N,h(0))$.
Suppose ({\bf C}) or ({\bf D}) is true, then the followings are true.
\begin{enumerate}
   \item [(i)] $f_t$ is area non-increasing for $t\in [0,T]$.
   \item [(ii)] If $f_0$ is strictly area decreasing at a point, then $f_t$ is strictly area decreasing for $t>0$. Moreover, there is a constant $a>0$ depending only on the bounds of the curvatures of $g(t), h(t), m, n$ so that $e^{at}\mathfrak{m}(t)$ is non-decreasing. In particular, if $g(t), h(t)$ are defined on $M\times[0,T_{max})$ and  $N\times[0,T_{max})$ respectively, then the mean curvature flow $F$ exists and remains graphic on $M\times[0,T_{max})$.
\item[(iii)] If in addition, $\chi^g+\chi^h> 0$ in case of ({\bf C}) and; $\chi^g>0$ in case of ({\bf D}),  then either  $f_t$ is strictly area decreasing for all $t\in (0,T]$ or  $f_0$ is a  Riemanian submersion (if $m>n$), local isometry (if $m=n$),  isometric immersion (if $m<n$).
\item[(iv)]  If in addition $\tau_N<0$ in case of ({\bf D}), then $f_t$ is strictly area decreasing for all $t\in (0,T]$.
 \end{enumerate}
\end{thm}
\begin{proof} We only prove the case ({\bf C}) while the case ({\bf D}) can be proved using similar argument. Let $\phi, \varphi$ be as in the proof the Theorem \ref{t-TTW-2} with $0\le \delta, \e<1$ and $L>0$ to be determined.
The proof of (i) and the first statement of (ii) are similar to the proof of Theorem \ref{t-TTW-2}.

We focus on the second assertion of (ii). Let $a>0$ to be determined and let $\ol \Theta =e^{at}\Theta$. We want to prove that  $e^{at}\mathfrak{m}(t)$ is nondecreasing. Since $\Theta>0$ for $t>0$, we may assume that $\mathfrak{m}(0)>0$. Suppose $e^{at}\mathfrak{m}(t)<\mathfrak{m}(0)$ for some $t>0$. Then there exists $x_0\in M, t_0>0$ such that $e^{at_0}\Theta_{1221}=e^{at_0}\mathfrak{m}(t_0)\le e^{at}\mathfrak{m}(t)$ for $0\le t\le t_0$. Then we by Lemmas \ref{l-positivity} (with $\a=0$) and Lemma~\ref{l-PIC-1},
 \bee
 \begin{split}
 0\ge& \heat\ol\Theta_{1221}+\frac12\Theta_{1221}^{-1}|\nabla\Theta_{1221}|^2\\
 \ge&-C_4\Theta_{1221}+a\Theta_{1221}
 \end{split}
 \eee
for some constant $C_4$ depending only on the bounds of the curvatures of $g(t), h(t), m, n$, at some point in space-time. This is impossible, if we take $a=2C_4$ because $\Theta>0$. Hence $e^{at}\mathfrak{m}(t)$ is nondecreasing. By Theorem \ref{thm:long-time-MCF-RF}, the last assertion of (ii) is true.

The proof of  (iii) and (iv) are similar to the proof of Theorem \ref{t-TTW-2}(iii) using the fact that if $\Theta_{1221}=0$ then $\lambda_1\ge 1$.
\end{proof}

\begin{rem}\label{rem:pre-IC-Ric3}
The conditions ({\bf C}) and ({\bf D}) are not necessarily preserved along the Ricci flow. When $n=3$, $\chi^g(t)\geq 0$ is equivalent to $\Ric\geq 0$ which is preserved along the Ricci flow by Hamilton \cite{Hamilton1982}. When $n\geq 4$, $\chi^g(t)\geq 0$ can be ensured by $\chi_{IC1}\geq 0$, which is preserved along the Ricci flow thanks to the work of Brendle-Schoen \cite{BrendleSchoen2009} and Nguyen \cite{Nguyen2010}.
\end{rem}

Now we are ready to prove the rigidity of maps under non-negative $1$-isotropic curvature condition.  We need the following:
\begin{lma}\label{l-long-time-RF}
Suppose $(M^n,g_0),n\geq 3$ is compact, simply connected, non-symmetric,  irreducible compact manifold such that  $\chi_{IC1}(g_0)\geq 0$. Let $g(t),t\in [0,T_{max})$ be the maximal Ricci flow solution starting from $g_0$. Then $T_{max}<\infty$ and as $t\to T_{\max}$, the curvature of $g(t)$ will tend to infinity. In particular, if $g_0$ is Einstein, then $g_0$ has positive sectional curvature.
When $n=3$, $\chi_{IC1}\geq 0$ is understood to be $\Ric\geq 0$.
\end{lma}

\begin{proof} If $n=3$, then  the result follows from   the work of Hamilton \cite{Hamilton1982,Hamilton1986} because $(M,g_0)$ is locally irreducible.

Suppose $n\ge 4$.
 $\chi_{IC1}(g_0)\geq 0$ implies that the scalar curvature $\mathcal{R}(g_0)\ge 0$. If $\mathcal{R}(g_0)\equiv0$, then $g_0$ is flat. In fact, for an orthonormal frame $e_i$, let $K_{ij}$ with $i\neq j$ be the sectional curvature of the two planes spanned by $e_i, e_j$. Then we have $K_{ik}+K_{il}=0$ for all $i<k<l$. Hence $K_{ij}=0$ for all $i<j$. Hence we must have $\mathcal{R}(g_0)>0$ somewhere. By the evolution equation of $\mathcal{R}(g(t))$ and the strong maximum principle, $\mathcal{R}(g(t))>0$ for $t>0$. Hence $T_{max}<\infty$.

First, for  $n\geq 4$,  it follows from \cite{BrendleSchoen2009,Nguyen2010} that $\chi_{IC1}(g(t))\geq 0$ for all $t\in [0,T_{max})$.   Suppose $\chi_{IC1}(g_0)>0$ at some point, then it follows from \cite{BrendleSchoen2009,Nguyen2010} that $\chi_{IC1}(g(t))>0$ for all $t>0$ and the result   follows from the work of Brendle \cite{Brendle2008}.

From now on, we assume this is not the case. We might assume $g=g(t_0)$ is irreducible by choosing $t_0$ sufficiently small.  We apply the Berger classification Theorem to deduce that $(M^n,g)$ is either quaternion-K\"ahler or has holonomy group $\mathrm{SO}(n)$ or $\mathrm{U}(\frac{n}{2})$ since the remaining are Ricci flat and hence flat by $\chi_{IC1}\geq 0$.   It also follows from \cite{Brendle2010} that quaternion-K\"ahler case is indeed symmetric and so does $g_0$. Hence,  the quaternion-K\"ahler case is ruled out.

Suppose $\mathrm{Hol}(M,g)=\mathrm{SO}(n)$, we claim that we must have $\chi_{IC1}(g)>0$.  This was implicitly proved in \cite{BrendleSchoen2008}, we include it for readers' convenience. Suppose there is $x_0\in M$, $\lambda\in [0,1]$ and orthonormal frame $\{e_i\}_{i=1}^4$ at $(x_0,t_0)$ such that
\begin{equation}
R_{1331}+\lambda^2 R_{1441}+R_{2332}+\lambda^2 R_{2442}-2\lambda R_{1234}=0.
\end{equation}
By \cite[Proposition 5]{BrendleSchoen2008}, the equality is invariant under parallel transport. Since $\mathrm{Hol}(M,g)=\mathrm{SO}(n)$, we might obtain $\mathcal{R}(g)=0$ at $x_0$ by considering the element $e_1\mapsto -e_1$, $(e_3,e_4)\mapsto (e_4,e_3)$ which is an element in $\mathrm{SO}(n)$ to show that
\begin{equation}
\left(R_{1331}+R_{1441}\right)+\left(R_{2332}+R_{2442}\right)=0.
\end{equation}
Since $\chi_{IC1}\geq 0$, we must have $R_{1331}+R_{1441}=0$. Using parallel transport with $\mathrm{Hol}(M,g)=\mathrm{SO}(n)$ again, we conclude that $\mathcal{R}(g)\equiv 0$ which is impossible. This proves our claim. Hence in this case, the lemma is true.

Suppose $\mathrm{Hol}(M,g)=\mathrm{U}(\frac{n}{2})$. Then by    \cite{Seshadri2009}, $(M,g)$ is \K with positive orthogonal bisectional curvature.   It follows from \cite{ChenTian2006,Chen2007,Wilking2013} that the normalized Ricci flow from $g(t)$ converges to $\mathbb{CP}^{n/2}$ as $t\to T_{max}$ after rescaling. Hence the curvature of $g(t)$ also tends to infinity.
\end{proof}

We now apply Theorem~\ref{t-pic} and Lemma \ref{l-long-time-RF} to study the rigidity of area non-increasing maps in the following two cases: \vskip0.2cm

({\bf E}): $(N^n,h_0)$  and $(M^m,g_0)$ satisfy
\begin{equation}
\left\{
\begin{array}{ll}
(N^n,h_0) \text{  is Einstein with } \kappa_N\geq 0;\\[3mm]
(M^m,g_0) \text{ is locally irreducible and non-symmetric with }\chi_{IC1}(g_0)\geq 0;\\[3mm]
\displaystyle  \mathcal{R}_{min}(g_0)\geq \frac{m}{n}\;\mathcal{R}_{max}(h_0)
\end{array}
\right.
\end{equation}

\vskip .2cm

({\bf F}): $(M^m,g_0)$ is locally irreducible and non-symmetric and $(N^n,h_0)$ satisfies
\be
\tau_N\leq 0, \quad \chi_{IC1}(g_0)\geq 0,\quad\text{and}\quad \mathcal{R}_{min}(g_0)\geq \frac{m}{n}\mathcal{R}_{max}(h_0)
\ee
where $\tau_N$ denotes an upper bound of the sectional curvature of $h_0$ and $\mathcal{R}$ denotes the scalar curvature.

\begin{thm}\label{t-pic-rigidity}
Suppose $(M^m, g_0)$ and $(N^n,h_0)$ be two compact manifolds with $\ell=\min\{m,n\}\geq 3$  such that ({\bf E}) or ({\bf F}) holds. If $f_0$ is a smooth map from $M$ to $N$ which is area non-increasing from $(M,g_0)$ to $(N,h_0)$, then we have the following.
\begin{enumerate}
\item[(i)]   $(M^m,g_0)$ is not Einstein, and $f_0$ is homotopically trivial; or
   \item [(ii)]   $(M^m,g_0)$ is Einstein, and either  $f_0$ is  homotopical trivial or  $f_0$ is local isometry (if $m=n$),  isometric immersion (if $m<n$).

\item[(iii)]  If in addition $\tau_N<0$ in case of ({\bf F}), then $f_0$ is homotopy trivial.
 \end{enumerate}
\end{thm}
\begin{proof}
We only prove the case ({\bf E}) while  the case ({\bf F}) is proved similarly.

(i) Suppose $g_0$ is not Einstein. Let us first assume that
$$
\mathcal{R}_{min}(g_0)> \frac{m}{n}\;\mathcal{R}_{max}(h_0).
$$
Then we can shrink $h_0$ to $a^2h_0$  for some $0<a<1$, so that $f_0$ is strictly area decreasing so that the above inequality on scalar curvatures is still true for $g_0, a^2h_0$. Hence without loss of generality, we may assume that $f_0$ is strictly area decreasing.
Let $g(t),t\in [0,T_{max})$ be the maximal solution of the Ricci flow starting from $g_0$. By Lemma \ref{l-long-time-RF}, $T_{max}<\infty$.
 Since $h_0$ is a Einstein metric,  $h(t)=(1-Lt)h_0$ for some $L\ge0$ where $\Ric(h_0)=Lh$. Hence $h(t)$ is defined on $[0, L^{-1}).$ Here $L^{-1}$ is understood to be $+\infty$ if $L=0$. By considering the lower bound of scalar curvature of $g(t)$, the strong maximum principle and the fact that $g_0$ is not Einstein, we have
\begin{equation}\label{e-nonEinstein}
\mathcal{R}(g(t))> \frac{mL}{1-Lt}=\frac mn\mathcal{R}(h(t))
\end{equation}
for $t\in (0,T_{max})$ and in particular $T_{\max}< L^{-1}$.  By Theorem \ref{t-pic}, we can solve the graphical mean curvature flow    $F: M\to (M\times N, g(t)\oplus h(t))$ with $F_0=\mathrm{Id}\times f_0$  which exists on $[0,T_{max})$. On the other hand, by Lemma \ref{l-long-time-RF}, the sectional curvature of $g(t)$ tends to infinity as $t\to T_{max}$ while the sectional curvature of $h(t)$ remains bounded in $[0,T_{max}]$ because $T_{max}<L^{-1}$. This reduces to the situation in Theorem \ref{t-TTW-2}, and hence $f_0$ is homotopically trivial.

If we only assume that $$
\mathcal{R}_{min}(g_0)\ge \frac{m}{n}\;\mathcal{R}_{max}(h_0),
$$
then   \eqref{e-nonEinstein} is still true for $t>0$ by strong maximum principle. Let $F$ be the short time solution of the graphical mean curvature flow as above,  by  Theorem~\ref{t-pic}, $f_t: (M,g(t))\to (N, h(t))$ is still area non-increasing. Moreover $g(t)$ is still in $\chi_{IC1}$ by \cite{BrendleSchoen2008}. Hence $f_t$ is homotopically trivial by the above discussion and hence $f_0$ is also homotopically trivial.

(ii)  Suppose $g_0$ is Einstein. Then the Einstein constant must be positive because $g_0$ is $\chi_{IC1}$ and is locally irreducible. So $g_0$ has positive Ricci curvature and hence its universal cover is compact. By Lemma \ref{l-long-time-RF}, $g_0$ has positive sectional curvature. If $m>n$, then by the assumption on the scalar curvatures of $g_0, h_0$
\bee
\Ric(g_0)_{min}-\Ric(g_0)_{max}+(m-n)\kappa_M>0
\eee
where $\kappa_M$ is the lower bound of the sectional curvature of $g_0$. We can shrink $h_0$ a little so that the above inequality is true and $f_0$ is strictly area decreasing. By Theorem \ref{t-TTW-2}(ii), we conclude that $f_0$ is homotopically trivial.

Since $g_0$ has positive sectional curvature, by Theorem \ref{t-TTW-2}(iii) we conclude that either $f_0$ is homotopically trivial or $f_0$ is a local isometry (if $m=n$) and isometric immersion (if $m<n$).

The proof of (iii) is similar, and we omit the details.

\end{proof}

\end{document}